\documentclass[final]{siamltex}
\usepackage{graphicx}
\usepackage{url}
\usepackage{mathptmx}     
\usepackage[english]{babel}
\usepackage{amssymb,amsmath}
\usepackage{enumitem}
\usepackage{color}
\usepackage{hyperref}
\usepackage{epstopdf}

\newtheorem{remark}[theorem]{Remark}

\newcommand{\e}{\mathrm{e}}

\newcommand{\E}{\mathbb{E}}
\newcommand{\bigo}[1]{\mathcal{O}(#1)}
\providecommand{\norm}[1]{\ensuremath{\lVert#1\rVert}}
\providecommand{\Norm}[1]{\ensuremath{\Big\lVert#1\Big\rVert}}
\providecommand{\tnorm}[1]{\lvert\lvert\lvert #1 \rvert\rvert\rvert} 
\newcommand{\abs}[1]{\left\vert#1\right\vert}
\newcommand{\dt}{\mathrm{d}t}
\newcommand{\Ph}{\mathcal{P}_h}
\newcommand{\Rh}{\mathcal{R}_h}
\newcommand{\dd}{\mathrm{d}}
\newcommand{\Tr}{\mathrm{Tr}}
\newcommand{\HS}{\mathrm{HS}}

\title{Full discretisation of semi-linear stochastic wave equations driven by multiplicative noise}

\author{Rikard Anton\thanks{Department of Mathematics and Mathematical
    Statistics, Ume{\aa} University, SE--901~87~Ume{\aa}, 
              Sweden ({\tt rikard.anton@umu.se}).}
       \and 
       David Cohen\thanks{Department of Mathematics and Mathematical
              Statistics, Ume{\aa} University, SE--901~87~Ume{\aa}, 
              Sweden ({\tt david.cohen@umu.se}). 
              Department of Mathematics, University of Innsbruck, 
              A--6020~Innsbruck, Austria  
              ({\tt david.cohen@uibk.ac.at})} 
       \and Stig Larsson\thanks{Department of Mathematical Sciences, 
              Chalmers University of Technology and University 
              of Gothenburg, SE--412 96 Gothenburg, Sweden
              ({\tt stig@chalmers.se}).} 
       \and Xiaojie Wang\thanks{School of Mathematics and Statistics,  
              Central South University, CN--410083 Changsha, Hunan, China
              ({\tt x.j.wang7@csu.edu.cn})}
         }

\begin{document}

\maketitle

\begin{abstract}
A fully discrete approximation of the
semi-linear stochastic wave equation driven 
by multiplicative noise is 
presented. A standard linear finite element  
approximation is used in space 
and a stochastic trigonometric method for the 
temporal approximation. 
This explicit time integrator allows 
for mean-square error bounds independent of the space 
discretisation and thus do not suffer from 
a step size restriction as in the often used 
St\"ormer-Verlet-leap-frog scheme. Furthermore, 
it satisfies an almost trace formula 
(i.\,e., a linear drift of the expected value 
of the energy of the problem).  
Numerical experiments are presented and 
confirm the theoretical results.
\end{abstract}



\pagestyle{myheadings}
\thispagestyle{plain}
\markboth{R.~Anton, D.~Cohen, S.~Larsson, and X.~Wang}{Full discretisation of stochastic wave equations}


\section{Introduction}\label{sect:intro}
We consider the numerical discretisation of semi-linear stochastic\\ wave equations of the form
\begin{equation}
\begin{aligned}
&\mathrm{d}\dot{u} - \Delta u\, \mathrm{d}t = f(u)\,\mathrm{d}t+g(u)\,\mathrm{d}W &&\quad \mathrm{in}\: \ 
\mathcal{D}\times(0,\infty), \\
&u = 0 &&\quad \mathrm{in}\: \ \partial\mathcal{D}\times(0,\infty), \\
&u(\cdot,0) = u_0, \ \dot{u}(\cdot,0)=v_0 &&\quad \mathrm{in}\: \ \mathcal{D},
\end{aligned}
\label{swe}
\end{equation}
where $u=u(x,t)$ and $\mathcal{D}\subset\mathbb{R}^d$, $d=1,2,3$, is a bounded convex 
domain with polygonal boundary $\partial\mathcal{D}$. 
The ``$\cdot$'' denotes the time derivative $\frac{\partial}{\partial t}$. 
Assumptions on the smoothness 
of the nonlinearities $f$ and $g$ will be given below. 
The stochastic process $\{W(t)\}_{t\geq 0}$ is an $L_2(\mathcal{D})$-valued 
(possibly cylindrical) $Q$-Wiener process with respect to a
normal filtration $\{\mathcal{F}_t\}_{t\geq 0}$ on a filtered
probability space $(\Omega,\mathcal{F},\mathbb{P},\{\mathcal{F}_t\}_{t\geq 0})$. 
The initial data $u_0$ and $v_0$ are $\mathcal{F}_0$-measurable random
variables. We will numerically solve this problem with a linear finite
element method in space and a stochastic trigonometric method in time.

We refer to the introductions of \cite{kls10} and \cite{cls13} 
for the relevant literature on the spatial, respectively temporal, 
discretisation of stochastic (linear) wave equations. 
Further, the recent publication \cite{wgt13} presents a full discretisation 
of the wave equation with additive noise: a spectral 
Galerkin approximation is used in space and an adapted 
stochastic trigonometric method, using linear 
functionals of the noise as in \cite{jk09}, 
is employed in time. Furthermore, the time 
discretisation of nonlinear stochastic wave equations by 
stochastic trigonometric methods is analysed in \cite{raey}. 
Finally, let us mention the recent publication \cite{cqs14} 
which analyses convergence in $L^p(\Omega)$ 
of the stochastic trigonometric method applied to the 
one-dimensional nonlinear stochastic wave equation. 

In the present publication, we prove mean-square convergence for the 
full discretisation to the exact solution to the nonlinear problem \eqref{swe}. 
Furthermore, using this result, we derive a geometric property of our 
numerical integrator, namely a trace formula. 
The trace formula (the linear drift of the expected value of the energy) 
for the exact solution of \eqref{swe} as well as for the finite element solution 
and the completely discrete solution are presented.  

Strong approximations of stochastic wave equations are relevant in many real applications. 
For example, let us consider the motion of a strand of DNA floating in a liquid as presented 
in \cite{MR2508773} and references therein. The motion of the DNA molecule may be modeled 
by a wave equation and the impact of the fluid's molecules may be modeled by a stochastic force acting 
on the string. When two normally distant parts of the DNA get close enough, biological events, 
such as release of enzymes, occur. It is thus of interest to consider strong 
approximation of stochastic wave equations in such a situation.

The paper is organised as follows. We introduce some notations and mention 
some useful results in the next section. Section~\ref{sect:ms}  
presents a mean-square convergence analysis for our numerical discretisation. 
A trace formula for the exact and numerical solutions is given in Section~\ref{sect:trace}. 
Finally, numerical experiments illustrating the rates of convergence and the trace formula 
of the numerical solution are given in the final section.


\section{Notations and useful results}\label{sect:not}
Let $U$ and $H$ be separable Hilbert spaces with norms $\norm{\cdot}_U$ and $\norm{\cdot}_H$ respectively. 
We denote the space of bounded linear operators from $U$ to $H$ by $\mathcal{L}(U,H)$, 
and we let $\mathcal{L}_2(U,H)$ be the set of Hilbert-Schmidt operators with norm
$$\norm{T}_{\mathcal{L}_2(U,H)}:=\left(\sum_{k=1}^\infty\norm{Te_k}_H^2\right)^{1/2},$$
where $\{e_k\}_{k=1}^\infty$ is an arbitrary orthonormal basis of $U$. 
If $H=U$, then we write $\mathcal{L}(U)=\mathcal{L}(U,U)$ and $\HS=\mathcal{L}_2(U,U)$. 
Let $Q\in\mathcal{L}(U)$ be a self-adjoint, positive semidefinite operator. 
We denote the space of Hilbert-Schmidt operators from $Q^{1/2}(U)$ to $H=U$ by $\mathcal{L}_2^0$ with norm
$$\norm{T}_{\mathcal{L}_2^0}=\norm{TQ^{1/2}}_{\HS}.$$
For the stochastic wave equation \eqref{swe}, we define $U:=L_2(\mathcal{D})$ 
and denote the $L_2(\mathcal{D})$-norm by $\norm{\cdot}:=\norm{\cdot}_{L_2(\mathcal{D})}$. 
Further, we set $\Lambda=-\Delta$ with $D(\Lambda)=H^2(\mathcal{D})\cap H^1_0(\mathcal{D})$.
 
Let $(\Omega,\mathcal{F},\mathbb{P},\{\mathcal{F}_t\}_{t\geq 0})$ be a filtered probability space 
and $L_2(\Omega,H)$ the space of $H$-valued square integrable random variables with norm
$$
\norm{v}_{L_2(\Omega,H)}:=\E[\norm{v}_H^2]^{1/2}.
$$
Next, we define the space $\dot{H}^\alpha=D(\Lambda^{\alpha/2})$, for $\alpha\in\mathbb{R},$ with norm
$$\norm{v}_\alpha:=\norm{\Lambda^{\alpha/2}v}_{L_2(\mathcal{D})}=
\left(\sum_{j=1}^\infty\lambda_j^\alpha(v,\varphi_j)_{L_2(\mathcal{D})}^2\right)^{1/2},$$
where $\{(\lambda_j,\varphi_j)\}_{j=1}^\infty$ are the eigenpairs of $\Lambda$ with orthonormal eigenvectors. 
We also introduce the space 
$$
H^\alpha:=\dot{H}^\alpha\times\dot{H}^{\alpha-1},
$$
with norm
$\tnorm{v}_\alpha^2:=\norm{v_1}_\alpha^2+\norm{v_2}_{\alpha-1}^2,$ for
$\alpha\in\mathbb{R}$ and $v=[v_1,v_2]^T$.  
Note that $\dot{H}^0=U:=L_2(\mathcal{D})$ and $H:=H^0=\dot H^0\times\dot H^{-1}$. 
In the following we denote the scalar product by 
$(\cdot,\cdot)=(\cdot,\cdot)_{L_2(\mathcal{D})}$ and recall the notation for the norm 
$\norm{\cdot}=\norm{\cdot}_{L_2(\mathcal{D})}$.

Denoting the velocity of the solution to our stochastic partial
differential equation by $u_2:=\dot u_1:=\dot u$,  
one can rewrite \eqref{swe} as
\begin{equation}
\begin{split}
& \mathrm{d}X(t) = AX(t)\,\mathrm{d}t+F(X(t))\,\mathrm{d}t+G(X(t))\,\mathrm{d}W(t), \ \ t>0, \\
& X(0) = X_0,
\label{swe2}
\end{split}
\end{equation}
where $X:=\begin{bmatrix} u_1\\ u_2 \end{bmatrix}$, 
$A:=\begin{bmatrix} 0 & I \\ -\Lambda & 0 \end{bmatrix}$, 
$F(X):=\begin{bmatrix} 0 \\ f(u_1) \end{bmatrix}$, 
$G(X):=\begin{bmatrix} 0 \\ g(u_1) \end{bmatrix}$ 
and $X_0:=\begin{bmatrix} u_0 \\v_0 \end{bmatrix}$. 
The operator $A$ with $D(A)=H^1=\dot H^1\times \dot H^{0}$ 
is the generator of a strongly continuous
semigroup of bounded linear operators 
$E(t)=\e^{tA}$ on $H=H^0=\dot{H}^0\times \dot{H}^{-1}$, in fact, 
a unitary group. 

Let $\{\mathcal{T}_h\}$ be a quasi-uniform family of triangulations of the convex polygonal domain 
$\mathcal{D}$ with $h_\mathrm{K}=\text{diam}(K)$ and $h=\max_{\mathrm{K}\in\mathcal{T}_h}h_\mathrm{K}$. 
Let $V_h\subset H_0^1(\mathcal{D}) = \dot{H}^1$ be the space 
of piecewise linear continuous functions with respect to $\mathcal{T}_h$ 
which are zero on the boundary of $\mathcal{D}$, and let $\mathcal{P}_h:\dot{H}^0\rightarrow V_h$ denote the $\dot{H}^0$-orthogonal projector and
  $\mathcal{R}_h:\dot{H}^1\rightarrow V_h$ the $\dot{H}^1$-orthogonal
  projector (Ritz projector).  Thus, 
  \begin{align*}
    (\mathcal{P}_hv,w_h)= (v,w_h), \quad
    (\nabla\mathcal{R}_hu,\nabla w_h)= (\nabla u,\nabla w_h), \quad 
   \forall v\in\dot{H}^0,\,u\in\dot{H}^1,\,w_h\in V_h.
  \end{align*}
 
The discrete Laplace operator $\Lambda_h\colon V_h\to V_h$ is then defined by
$$
(\Lambda_hv_h,w_h)=(\nabla v_h,\nabla w_h)\quad \forall w_h\in V_h.
$$
We note that $\mathcal{R}_h=\Lambda_h^{-1}\mathcal{P}_h\Lambda$. We also define
discrete variants of $\norm{\cdot}_\alpha$ and $\dot{H}^\alpha$ by 
$$
\norm{v_h}_{h,\alpha}=\norm{\Lambda_h^{\alpha/2}v_h},\quad v_h\in V_h
$$
and $\dot{H}_h^\alpha=V_h$ equipped with the norm $\norm{\cdot}_{h,\alpha}$. 
Finally, the finite element approximation of \eqref{swe} can then be written as
\begin{equation}
\begin{split}
& \mathrm{d}\dot{u}_{h,1}(t) + \Lambda_h u_{h,1}(t)\,\mathrm{d}t = 
\mathcal{P}_h\,f(u_{h,1}(t))\,\dt+\mathcal{P}_h\,g(u_{h,1}(t))\,\mathrm{d}W(t), \ \ t>0, \\
& u_{h,1}(0)=u_{h,0}, \ u_{h,2}(0)=v_{h,0},
\label{femswe1}
\end{split}
\end{equation}
or in the abstract form
\begin{equation}
\begin{split}
& \mathrm{d}X_h(t) = A_hX_h(t)\,\mathrm{d}t+\Ph F(X_h(t))\,\dt
+\Ph G(X_h(t))\,\mathrm{d}W(t), \ \ t>0, \\
& X_h(0) = X_{h,0},
\label{femswe2}
\end{split}
\end{equation}
where $ A_h := \begin{bmatrix} 0 & I \\ -\Lambda_h & 0 \end{bmatrix}$,
$X_h := \begin{bmatrix} u_{h,1} \\ u_{h,2} \end{bmatrix}$, $F$ and $G$
are as before, and $X_{h,0} := \begin{bmatrix} u_{h,0} \\ v_{h,0} \end{bmatrix}$
with $u_{h,0}=\Rh u_0,\ v_{h,0}=\Ph v_0\in V_h$. Note the abuse of notation for the projection 
$\Ph F(X_h)=(0,\Ph f(u_{h,1}))^T$ and similarly for $\Ph G(X_h)$. This will be used throughout the paper.
Again, $A_h$ is the generator of a $C_0$-semigroup $E_h(t)=\e^{tA_h}$ 
on $H_h:=\dot{H}_h^0\times \dot{H}_h^{-1}$. 

We study the equations \eqref{swe2} and \eqref{femswe2} in their mild form
\begin{align}
X(t)&=E(t)X_0+\int_0^tE(t-s)F(X(s))\,\mathrm{d}s+\int_0^tE(t-s)G(X(s))\,\mathrm{d}W(s),\label{exactsol}\\
X_h(t)&=E_h(t)X_{h,0}+\int_0^tE_h(t-s)\Ph F(X_h(s))\,\mathrm{d}s+
\int_0^tE_h(t-s)\Ph G(X_h(s))\,\mathrm{d}W(s),\label{exactsolfem}
\end{align}
where the semigroups can be expressed as 
\begin{align}
E(t)&=\begin{bmatrix} C(t) & \Lambda^{-1/2}S(t) \\ -\Lambda^{1/2}S(t) & C(t) \end{bmatrix},\label{E}\\
E_h(t)&=\begin{bmatrix} C_h(t) & \Lambda_h^{-1/2}S_h(t) \\ -\Lambda_h^{1/2}S_h(t) & C_h(t) \end{bmatrix},
\label{E_h}
\end{align}
with $C(t)=\cos(t\Lambda^{1/2})$, $S(t)=\sin(t\Lambda^{1/2})$, $C_h(t)=\cos(t\Lambda_h^{1/2})$ 
and $S_h(t)=\sin(t\Lambda_h^{1/2})$.

In order to ensure existence and uniqueness of problem \eqref{swe} we shall assume that 
$u_0\in L_2(\Omega,\dot{H}^{\gamma})$ and $v_0\in
L_2(\Omega,\dot{H}^{\gamma-1})$, with $\gamma=\max(\beta,1)$  
for some regularity parameter $\beta\geq 0$, 
and that the functions $f\colon L_2(\mathcal{D})\to L_2(\mathcal{D})$ 
and $g\colon L_2(\mathcal{D})\to \mathcal{L}_2^0$ satisfy 
\begin{equation}
\begin{aligned}\label{assFG}
\norm{f(u)-f(v)}+\norm{g(u)-g(v)}_{\mathcal{L}_2^0}&\leq C \norm{u-v},&&\quad \text{if}\: \ \beta\geq 0,\\
\norm{f(u)}+\norm{\Lambda^{(\beta-1)/2}g(u)}_{\mathcal{L}_2^0} &\leq C (1+\norm{u} ),&&\quad\text{if}\: \ 0\leq\beta\leq1,\\
\norm{\Lambda^{(\beta-1)/2}f(u)}+\norm{\Lambda^{(\beta-1)/2}g(u)}_{\mathcal{L}_2^0} 
&\leq C (1+\norm{\Lambda^{(\beta-1)/2}u} ),&&\quad\text{if}\: \ \beta>1,
\end{aligned}
\end{equation}
for all $u,v\in L_2(\mathcal{D})$ in the first two inequalities and
for all $u\in \dot{H}^{\beta-1}$ in the last one. 
Through the text, $C$ (or $C_1, C_2, K_1, K_2$ etc.) denotes a generic positive constant 
that may vary from line to line.  We assume that the order of initial
  regularity $\gamma\ge1$ so that the discrete initial value
  $u_{h,0}=\Rh u_0$ is well defined. 

\begin{lemma}
Assume that $u_0\in L_2(\Omega,\dot{H}^{\gamma})$, $v_0\in
L_2(\Omega,\dot{H}^{\gamma-1})$ with $\gamma=\max(\beta,1)$ 
and the functions $f$ and $g$ satisfy \eqref{assFG} for some $\beta\geq0$.  
Then there exists a unique solution to the stochastic wave equation \eqref{swe2} 
and the finite element equation \eqref{femswe2} given by the solution 
of their respective mild equation, i.\,e., equations \eqref{exactsol} and \eqref{exactsolfem}.
\end{lemma}

The proof of this lemma follows from \cite[Theorem 7.4]{DaPrato1992}, see also the proof of Theorem~2.1 in \cite{raey}.

We now collect some results that we will use later on. Sketches of the proofs of these results 
are collected in the appendix at the end of this paper.

\noindent $\bullet$ The error estimates for the cosine and sine operators (Corollary~4.2 in \cite{kls10}):
Denote $X_0=[u_0,v_0]^T$ and let 
\begin{align*}
{\mathcal{G}_h(t)}X_0&=\bigl(C_h(t)\Rh-C(t)\bigr)u_0+\bigl(\Lambda_h^{-1/2}S_h(t)\Ph-\Lambda^{-1/2}S(t)\bigr)v_0,\\
\dot{{\mathcal{G}}}_h(t)X_0&=-\bigl(\Lambda_h^{1/2}S_h(t)\Rh-\Lambda^{1/2}S(t)\bigr)u_0+\bigl(C_h(t)\Ph-C(t)\bigr)v_0.
\end{align*}
Then we have
\begin{align}\label{initialerror1}
\begin{aligned}
\norm{{\mathcal{G}}_h(t)X_0}&\leq C\cdot(1+t)\cdot h^{\gamma-1}\tnorm{X_0}_\gamma,&&\quad t\geq0,\quad\gamma\in[1,3], \\
\norm{\dot{{\mathcal{G}}}_h(t)X_0}&\leq C\cdot(1+t)\cdot h^{\frac23(\gamma-1)}\tnorm{X_0}_\gamma,&&\quad t\geq0,\quad\gamma\in[1,4].
\end{aligned}
\end{align}
These will be used to estimate the error contributions from the
initial values.   In order to deal with the convolution terms in
\eqref{exactsolfem} we single out the following error estimates.  Let 
\begin{align*}
{\mathcal{K}_h(t)}v_0&=\bigl(\Lambda_h^{-1/2}S_h(t)\Ph-\Lambda^{-1/2}S(t)\bigr)v_0,\\
\dot{{\mathcal{K}}}_h(t)v_0&=\bigl(C_h(t)\Ph-C(t)\bigr)v_0.
\end{align*}
Then we have
\begin{align}\label{initialerror}
\begin{aligned}
\norm{{\mathcal{K}}_h(t)v_0}&\leq C\cdot(1+t)\cdot h^{\frac23\beta}\norm{v_0}_{\beta-1},&&\quad t\geq0,\quad\beta\in[0,3], \\
\norm{\dot{{\mathcal{K}}}_h(t)v_0}&\leq C\cdot(1+t)\cdot h^{\frac23(\beta-1)}\norm{v_0}_{\beta-1},&&\quad t\geq0,\quad\beta\in[1,4].
\end{aligned}
\end{align}

\noindent $\bullet$  The temporal H\"{o}lder continuity of the sine and cosine operators, 
see $(4.1)$ in \cite{cls13}:
\begin{equation}\label{o41cls}
\begin{aligned}
\norm{(S_h(t)-S_h(s))\Lambda_h^{-\beta/2}}_{\mathcal{L}(U)}&\leq C\cdot\abs{t-s}^\beta,&&\quad\beta\in[0,1],\\
\norm{(C_h(t)-C_h(s))\Lambda_h^{-(\beta-1)/2}}_{\mathcal{L}(U)}&\leq C\cdot\abs{t-s}^{\beta-1},&&\quad\beta\in[1,2],
\end{aligned}
\end{equation}
together with its continuous version:
\begin{equation}\label{41cls}
\begin{aligned}
\norm{(S(t)-S(s))\Lambda^{-\beta/2}}_{\mathcal{L}(U)}&\leq C\cdot\abs{t-s}^\beta,&&\quad\beta\in[0,1],\\
\norm{(C(t)-C(s))\Lambda^{-(\beta-1)/2}}_{\mathcal{L}(U)}&\leq C\cdot\abs{t-s}^{\beta-1},&&\quad\beta\in[1,2].
\end{aligned}
\end{equation}

\noindent $\bullet$ The equivalence of $\Lambda_h$ and $\Lambda$, 
see the proof of Theorem~4.4 in \cite{kll}: This uses an inverse inequality, 
hence our assumption about the quasi-uniformity of the mesh family.
\begin{equation}
\norm{\Lambda_h^\alpha\mathcal{P}_h\Lambda^{-\alpha}v}^2\leq \norm{v}^2, 
\ \ \alpha\in[-\tfrac12,1], \ \ v\in \dot{H}^0 = L_2(\mathcal{D}).
\label{lpl}
\end{equation}

\noindent $\bullet$ The equivalence of the discrete and continuous norm, see $(2.13)$ in \cite{al13}: 
\begin{align}\label{equivnorm}
c\norm{\Lambda_h^\gamma v_h}\leq\norm{\Lambda^\gamma v_h}
\leq C\norm{\Lambda_h^\gamma v_h}\quad\text{for}\quad v_h\in V_h\quad \text{and} \quad\gamma\in[-\tfrac12,\tfrac12].
\end{align}

Using the above estimates, one can deduce the following regularity results for 
the exact solution to our stochastic wave equation \eqref{swe} and 
for the exact solution of the finite element approximation \eqref{femswe1}.
\begin{proposition}\label{prop:regEx}
Let $[u_1,u_2]^T$ be the solution to \eqref{swe}, where the initial values satisfy $u_0\in L_2(\Omega,\dot{H}^{\gamma})$, 
$v_0\in L_2(\Omega,\dot{H}^{\gamma-1})$ with $\gamma=\max(\beta,1)$, and the functions $f$ and $g$ satisfy \eqref{assFG} for some $\beta\geq0$. 
Then it holds that
\begin{align*}
\sup_{0\leq t\leq T}\E[\norm{u_1(t)}_{\beta}^2+\norm{u_2(t)}_{\beta-1}^2]\leq C 
\end{align*}
and, for $0\leq s\leq t\leq T$, 
\begin{align*}
\E[\norm{u_1(t)-u_1(s)}^2]&\leq C\abs{t-s}^{2\min(\beta,1)}
\Bigl(\E[\norm{u_0}_\beta^2+\norm{v_0}_{\beta-1}^2]
\\ &\quad+\sup_{r\in[0,T]}\E[1+\norm{u_1(r)}_{\beta}^2]\Bigr).
\end{align*}
\end{proposition}

The proof of this proposition is very similar to the proof of Proposition~\ref{prop:regFEM} given 
below and is therefore omitted (see also the proofs of Proposition~3.1 and Lemma~3.3 in \cite{raey}). 

The next result will be useful in Section~\ref{sect:trace} when we will deal 
with the trace formula of the numerical solution. 

\begin{proposition}\label{prop:regFEM}
Let $[u_{h,1},u_{h,2}]^T$ be the solution to the finite element
problem \eqref{femswe1}, where the initial values  
satisfy $u_0\in L_2(\Omega,\dot{H}^{\gamma})$, 
$v_0\in L_2(\Omega,\dot{H}^{\gamma-1})$ with $\gamma=\max(\beta,1)$, 
and the functions $f$ and $g$ satisfy \eqref{assFG} for some $\beta\in[0,2]$. 
Then it holds that
\begin{align*}
\sup_{0\leq t\leq T}\E[\norm{u_{h,1}(t)}_{h,\beta}^2+\norm{u_{h,2}(t)}_{h,\beta-1}^2]\leq C 
\end{align*}
and for $0\leq s\leq t\leq T$ 
\begin{align*}
\E[\norm{u_{h,1}(t)-u_{h,1}(s)}^2]&\leq C\abs{t-s}^{2\min(\beta,1)}\Bigl(\E[\norm{u_{h,0}}_{h,\beta}^2
+\norm{v_{h,0}}_{h,\beta-1}^2]\\
&\quad+\sup_{r\in[0,T]}\E[1+\norm{u_{h,1}(r)}_{h,\beta}^2]\Bigr),
\end{align*}
where we recall that $u_{h,0}$ and $v_{h,0}$ are the initial position and velocity to the finite element problem. 
\end{proposition}

\begin{proof}
Let us start with the first estimate of the norm of $\Lambda_h^{\beta/2}u_{h,1}(t)$ and consider the expression 
\begin{align*}
\Lambda_h^{\beta/2}u_{h,1}(t)&=\Lambda_h^{\beta/2}C_h(t)u_{h,0}+\Lambda_h^{(\beta-1)/2}S_h(t)v_{h,0}\\
&\quad+\int_0^t\Lambda_h^{(\beta-1)/2}S_h(t-r)\Ph f(u_{h,1}(r))\,\dd r\\
&\quad+\int_0^t\Lambda_h^{(\beta-1)/2}S_h(t-r)\Ph g(u_{h,1}(r))\,\dd W(r).
\end{align*}
Using the fact that $\Lambda_h$ and $C_h(t)$ commute, the boundedness of the cosine operator, 
together with our assumptions on the initial values for the finite element problem, we get 
\begin{align*}
\E[\norm{\Lambda_h^{\beta/2}C_h(t)u_{h,0}}^2]\leq C\quad\text{for}\quad \beta\in[0,2].
\end{align*}
Similarly, one obtains
$$
\E[\norm{\Lambda_h^{(\beta-1)/2}S_h(t)v_{h,0}}^2]\leq C.
$$
To estimate the third term, we use \eqref{lpl}, 
the assumptions on $f$ given in \eqref{assFG}, and the equivalence of the norms 
stated in \eqref{equivnorm}. First for $\beta\in[0,1]$, we get
\begin{align*}
&\E\Big[\Norm{\int_0^t\Lambda_h^{(\beta-1)/2}S_h(t-r)\Ph f(u_{h,1}(r))\,\dd r}^2\Big]\\
&\quad\leq C_1+C_2\int_0^t\E[\norm{u_{h,1}(r)}^2]\,\dd r\\
&\quad\leq C_3+C_4\int_0^t\E[\norm{u_{h,1}(r)}_{h,\beta}^2]\,\dd r,  
\end{align*}
because $S_h(t)$ and $\Lambda_h^{-(1-\beta)/2}$ are bounded. For $\beta\in[1,2]$, we have by \eqref{lpl}
\begin{align*}
&\E\Big[\Norm{\int_0^tS_h(t-r)\Lambda_h^{(\beta-1)/2}\Ph\Lambda^{-(\beta-1)/2}\Lambda^{(\beta-1)/2}
f(u_{h,1}(r))\,\dd r}^2\Big]\\
&\quad\leq C\int_0^t \E[1+\norm{\Lambda^{(\beta-1)/2}u_{h,1}(r)}^2]\,\dd r \leq 
C_1+C_2\int_0^t\E[\norm{u_{h,1}(r)}_{h,\beta-1}^2]\,\dd r \\
&\quad\leq C_3+C_4\int_0^t\E[\norm{u_{h,1}(r)}_{h,\beta}^2]\,\dd r. 
\end{align*}
Finally, Ito's isometry, equations \eqref{equivnorm} and \eqref{lpl}, 
and the assumptions \eqref{assFG} on $g$ give us
\begin{align*}
\E\Big[\Norm{\int_0^t S_h(t-r)\Ph\Lambda^{-(\beta-1)/2}\Lambda^{(\beta-1)/2}
  g(u_{h,1}(r))\,\dd W(r)}^2\Big]
\leq 
C_3+C_4\int_0^t\E[\norm{u_{h,1}(r)}_{h,\beta}^2]\,\dd r. 
\end{align*}
All together, for $\beta\in[0,2]$, one thus obtains
$$
\E[\norm{u_{h,1}(t)}_{h,\beta}^2]\leq K_1+K_2\int_0^t\E[\norm{u_{h,1}(r)}_{h,\beta}^2]\,\dd r
$$
and an application of Gronwall's lemma give the desired bound for $\E[\norm{u_{h,1}(t)}_{h,\beta}^2]$. 

The proof for the other bound is done in the same way except for a slight difference in the initial values 
and that $\Lambda_h^{(\beta-1)/2}S_h(t-r)$ in the integrals is replaced by $\Lambda_h^{(\beta-1)/2}C_h(t-r)$.

We now prove a H\"older regularity property of the finite element solution. 
We write, for $0\leq s\leq t\leq T$,
\begin{align*}
u_{h,1}(t)-u_{h,1}(s)&=(C_h(t)-C_h(s))u_{h,0}+\Lambda_h^{-1/2}(S_h(t)-S_h(s))v_{h,0}\\
&\quad+\int_0^s\Lambda_h^{-1/2}(S_h(t-r)-S_h(s-r))\Ph f(u_{h,1}(r))\,\dd r\\
&\quad+\int_s^t\Lambda_h^{-1/2}S_h(t-r)\Ph f(u_{h,1}(r))\,\dd r\\
&\quad+\int_0^s\Lambda_h^{-1/2}(S_h(t-r)-S_h(s-r))\Ph g(u_{h,1}(r))\,\dd W(r)\\
&\quad+\int_s^t\Lambda_h^{-1/2}S_h(t-r)\Ph g(u_{h,1}(r))\,\dd W(r).
\end{align*}
To estimate the first term we use \eqref{o41cls} to get
\begin{align*}
\E[\norm{(C_h(t)-C_h(s))u_{h,0}}^2]&=\E[\norm{(C_h(t)-C_h(s))\Lambda_h^{-\beta/2}\Lambda_h^{\beta/2}u_{h,0}}^2]\\
&\leq C|t-s|^{2\beta}\E[\norm{\Lambda_h^{\beta/2}u_{h,0}}^2],
\end{align*}
for $\beta\in[0,1]$. For $\beta\in(1,2]$ we note that $\Lambda_h^{-\beta/2}=\Lambda_h^{-1/2}\Lambda_h^{-(\beta-1)/2}$ 
and that
$\Lambda_h^{-(\beta-1)/2}$ is bounded in the operator norm.
Using a similar argument for the second term, we get the following estimate for the first two terms
\begin{align*}
&\E[\norm{(C_h(t)-C_h(s))u_{h,0}+\Lambda_h^{-1/2}(S_h(t)-S_h(s))v_{h,0}}^2]\\
&\quad\leq C|t-s|^{2\min(\beta,1)}\E[\norm{u_{h,0}}_{h,\beta}^2
+\norm{v_{h,0}}_{h,\beta-1}^2],
\end{align*}
for $\beta\in[0,2]$.
In order to estimate the third term, we use \eqref{o41cls}, 
the assumptions on $f$, and the equivalence of the norms given 
in \eqref{equivnorm}. First for $\beta\in[0,1]$, we obtain 
\begin{align*}
&\E\Big[\Norm{\int_0^s\Lambda_h^{-1/2}(S_h(t-r)-S_h(s-r))\Ph f(u_{h,1}(r))\,\dd r}^2\Big]\\
&\qquad\leq C|t-s|^{2}\sup_{t\in[0,T]}\E[1+\norm{u_{h,1}(t)}^2]\\
&\qquad\leq C|t-s|^{2}\sup_{t\in[0,T]}\E[1+\norm{u_{h,1}(t)}_{h,\beta}^2].
\end{align*}
For $\beta\in[1,2]$ we have, using \eqref{o41cls}, \eqref{lpl}, \eqref{equivnorm} 
and the fact that $\Lambda_h^{-(\beta-1)/2}$ is bounded in the operator norm
\begin{align*}
&\E\Big[\Norm{\int_0^s\Lambda_h^{-1/2}(S_h(t-r)-S_h(s-r))\Ph f(u_{h,1}(r))\,\dd r}^2\Big]\\
&\qquad\leq\int_0^s\E[\norm{(S_h(t-r)-S_h(s-r))\Lambda_h^{-1/2}\Lambda_h^{-(\beta-1)/2}
\Lambda_h^{(\beta-1)/2}\Ph\Lambda^{-(\beta-1)/2}\\
&\qquad\qquad\times\Lambda^{(\beta-1)/2} f(u_{h,1}(r))}^2]\,\dd r\\
&\qquad\leq C|t-s|^2\sup_{t\in[0,T]}\E[1+\norm{u_{h,1}(t)}_{h,\beta-1}^2]\\
&\qquad\leq C|t-s|^2\sup_{t\in[0,T]}\E[1+\norm{u_{h,1}(t)}_{h,\beta}^2] .
\end{align*}
Similarly we get for the fourth term
\begin{align*}
&\E\Big[\Norm{\int_s^t\Lambda_h^{-1/2}S_h(t-r)\Ph f(u_{h,1}(r))\,\dd r}^2\Big]\\
&\qquad\leq C|t-s|^{2\min(\beta,1)}\sup_{t\in[0,T]}\E[1+\norm{u_{h,1}(t)}_{h,\beta}^2].
\end{align*}
To estimate terms five and six we use Ito's isometry, \eqref{o41cls}, \eqref{lpl}, \eqref{equivnorm} and 
the assumptions on $g$ to get, 
for $\beta\in[0,1]$,
\begin{align*}
&\E\Big[\Norm{\int_0^s\Lambda_h^{-1/2}(S_h(t-r)-S_h(s-r))\Ph g(u_{h,1}(r))\,\dd W(r)}^2\Big]\\
&\quad\leq \int_0^s\E[\norm{(S_h(t-r)-S_h(s-r))\Lambda_h^{-\beta/2}\Lambda_h^{(\beta-1)/2}\Ph\\
&\qquad\times g(u_{h,1}(r))}_{\mathcal{L}_2^0}^2]\,\dd r\\
&\quad\leq C|t-s|^{2\beta}\sup_{t\in[0,T]}\E[\norm{\Lambda_h^{(\beta-1)/2}g(u_{h,1}(t))}_{\mathcal{L}_2^0}^2]\\
&\quad\leq C|t-s|^{2\beta}\sup_{t\in[0,T]}\E[1+\norm{u_{h,1}(t)}_{h,\beta}^2]
\end{align*}
and
\begin{align*}
&\E\Big[\Norm{\int_s^t\Lambda_h^{-1/2}S_h(t-r)\Ph g(u_{h,1}(r))\,\dd W(r)}^2\Big]\\
&\quad\leq \int_s^t\E[\norm{S_h(t-r)\Lambda_h^{-\beta/2}\Lambda_h^{(\beta-1)/2}\Ph 
g(u_{h,1}(r))}_{\mathcal{L}_2^0}^2]\,\dd r\\
&\quad\leq C|t-s|^{2\beta}\sup_{t\in[0,T]}\E[\norm{\Lambda_h^{(\beta-1)/2}g(u_{h,1}(t))}_{\mathcal{L}_2^0}^2]\\
&\quad\leq C|t-s|^{2\beta}\sup_{t\in[0,T]}\E[1+\norm{u_{h,1}(t)}_{h,\beta}^2].
\end{align*}
For $\beta\in[1,2]$ we again use that $\Lambda_h^{-(\beta-1)/2}$ is bounded in the operator norm. 

Collecting the above estimates give us the statement about the regularity of the finite element solution.
\end{proof}

\section{Mean-square convergence analysis}\label{sect:ms}

Recall that the exact solutions to \eqref{swe2} and \eqref{femswe2} solve the following equations
\begin{align*}
X(t)&=E(t)X_0+\int_0^tE(t-s)F(X(s))\,\mathrm{d}s+\int_0^tE(t-s)G(X(s))\,\mathrm{d}W(s),\nonumber\\
X_h(t)&=E_h(t)X_{h,0}+\int_0^tE_h(t-s)\Ph F(X_h(s))\,\mathrm{d}s+
\int_0^tE_h(t-s)\Ph G(X_h(s))\,\mathrm{d}W(s),\nonumber
\end{align*}
where $X_0=
\begin{bmatrix}
  u_0\\v_0
\end{bmatrix}
$, 
$X_{h,0}=
\begin{bmatrix}
  \Rh u_0\\ \Ph v_0
\end{bmatrix}
$ and  
\begin{align*}
E(t)=\begin{bmatrix} C(t) & \Lambda^{-1/2}S(t) \\ -\Lambda^{1/2}S(t) & C(t) \end{bmatrix},
\quad E_h(t)=\begin{bmatrix} C_h(t) & \Lambda_h^{-1/2}S_h(t) \\ -\Lambda_h^{1/2}S_h(t) & C_h(t) \end{bmatrix},
\end{align*}
with $C(t)=\cos(t\Lambda^{1/2})$, $S(t)=\sin(t\Lambda^{1/2})$, $C_h(t)=\cos(t\Lambda_h^{1/2})$ 
and $S_h(t)=\sin(t\Lambda_h^{1/2})$.

The explicit time discretisation of the finite element solution \eqref{femswe2} of the stochastic wave equation 
using a stochastic trigonometric method with stepsize $k$ reads 
$$
U^{n+1}=E_h(k)U^n+E_h(k)\Ph F(U^n)k+E_h(k)\Ph G(U^n)\Delta W^n,
$$ 
that is, 
\begin{align}
\begin{bmatrix} U_1^{n+1} \\ U_2^{n+1} \end{bmatrix} &= 
\begin{bmatrix} C_h(k) & \Lambda_h^{-1/2}S_h(k) \\ -\Lambda_h^{1/2}S_h(k) & C_h(k) \end{bmatrix} 
\begin{bmatrix} U_1^n \\ U_{2}^n \end{bmatrix} 
+\begin{bmatrix} \Lambda_h^{-1/2}S_h(k) \\ C_h(k) \end{bmatrix}
\mathcal{P}_hf(U_1^n)k\nonumber\\
&\quad+\begin{bmatrix} \Lambda_h^{-1/2}S_h(k) \\ C_h(k) \end{bmatrix}
\mathcal{P}_hg(U_1^n)\Delta W^n,
\label{stm}
\end{align}
where $\Delta W^n=W(t_{n+1})-W(t_n)$ denotes the Wiener
increments. Here we thus get an approximation $U_j^n\approx
u_{h,j}(t_n)$ of the exact solution of our finite element problem 
at the discrete times $t_n=nk$. Further, a recursion gives 
$$
U^n=E_h(t_n)U^0+\sum_{j=0}^{n-1}E_h(t_n-t_j)\Ph F(U^j)\,k
+\sum_{j=0}^{n-1}E_h(t_n-t_j)\Ph G(U^j)\,\Delta W^j.
$$
We now look at the error between the numerical and the exact solutions $U^n-X(t_n)$. 
We follow the same approach as in \cite{y05} for parabolic problems, see also \cite{lt13}, and obtain
\begin{align*}
\E[\norm{U^n-X(t_n)}^2]\leq 3\bigl(\E[\norm{\text{Err}_{0}}^2]+\E[\norm{\text{Err}_{\text d}}^2]+\E[\norm{\text{Err}_{\text s}}^2]\bigr),
\end{align*}
where we define
\begin{align*}
\text{Err}_{0}:=(E_h(t_n)\Ph-E(t_n))X_0,
\end{align*}
\begin{align*}
\text{Err}_{\text d}:=\sum_{j=0}^{n-1}\int_{t_j}^{t_{j+1}}
\Bigl(E_h(t_n-t_j)\Ph F(U^j)-E(t_n-s)F(X(s))\Bigr)\,\text ds
\end{align*}
and
\begin{align*}
\text{Err}_{\text s}:=\sum_{j=0}^{n-1}\int_{t_j}^{t_{j+1}}
\Bigl(E_h(t_n-t_j)\Ph G(U^j)-E(t_n-s)G(X(s))\Bigr)\,\text dW(s).
\end{align*}
We next estimate the above three terms.

{\bf Estimate for the initial error $\text{Err}_{0}$}. By \eqref{initialerror1}, the first component reads
\begin{align*}
&\E[\norm{(C_h(t_n)\Rh-C(t_n))u_0+(\Lambda_h^{-1/2}S_h(t_n)\Ph-\Lambda^{-1/2}S(t_n))v_0}^2]\\
&\quad\leq C(1+t_n)^2h^{2(\gamma-1)}\left(\E[\norm{u_0}_\gamma+\norm{v_0}_{\gamma-1}]\right)^2,
\end{align*}
for $\gamma\in[1,3]$.
Similarly for the second component
\begin{align*}
&\E[\norm{-(\Lambda_h^{1/2}S_h(t_n)\Rh-\Lambda^{1/2}S(t_n))u_0+(C_h(t_n)\Ph-C(t_n))v_0}^2]\\
&\quad\leq C(1+t_n)^2h^{\frac{4}{3}(\gamma-1)}\left(\E[\norm{u_0}_\gamma+\norm{v_0}_{\gamma-1}]\right)^2,
\end{align*}
for $\gamma\in[1,4]$.

{\bf Estimate for the deterministic part, $\text{Err}_{\text d}$}. We write the deterministic error as
\begin{align*}
\text{Err}_{\text d}&=\sum_{j=0}^{n-1}\int_{t_j}^{t_{j+1}}
\Bigl(E_h(t_n-t_j)\Ph F(U^j)-E(t_n-s)F(X(s))\Bigr)\,\dd s\\
&=\sum_{j=0}^{n-1}\int_{t_j}^{t_{j+1}}E_h(t_n-t_j)\Ph(F(U^j)-F(X(t_j)))\,\dd s\\
&\quad+\sum_{j=0}^{n-1}\int_{t_j}^{t_{j+1}}E_h(t_n-t_j)\Ph(F(X(t_j))-F(X(s)))\,\dd s\\
&\quad+\sum_{j=0}^{n-1}\int_{t_j}^{t_{j+1}}
\bigl(E_h(t_n-t_j)\Ph-E(t_n-t_j)\bigr)F(X(s))\,\dd s\\
&\quad+\sum_{j=0}^{n-1}\int_{t_j}^{t_{j+1}}\bigl(E(t_n-t_j)-E(t_n-s)\bigr)F(X(s))\,\dd s\\
&=:I_1+I_2+I_3+I_4,
\end{align*}
and estimate the second moment of each term in the above equation.
For the first component of the first term we get the following
estimate by using \eqref{41cls} and \eqref{assFG}
\begin{align*}
\left(\E[\norm{I_{[1,1]}}^2]\right)^{1/2}&\leq\sum_{j=0}^{n-1}\int_{t_j}^{t_{j+1}}\left(\E[\norm{\Lambda_h^{-1/2}S_h(t_n-t_j)\Ph(f(U_1^j)-f(u_1(t_j)))}^2]\right)^{1/2}\,\dd s\\
&\leq C\sum_{j=0}^{n-1}k\left(\E[\norm{U_1^j-u(t_j)}^2]\right)^{1/2},
\end{align*}
so that 
\begin{align*}
\E[\norm{I_{[1,1]}}^2]\leq \left(Ck\sum_{j=0}^{n-1}\left(\E[\norm{U_1^j-u(t_j)}^2]\right)^{1/2}\right)^2
\leq Ck\sum_{j=0}^{n-1}\E[\norm{U_1^j-u(t_j)}^2].
\end{align*}
The second component is estimated in the same way
\begin{align*}
\left(\E[\norm{I_{[1,2]}}^2]\right)^{1/2}&\leq\sum_{j=0}^{n-1}\int_{t_j}^{t_{j+1}}\left(\E[\norm{C_h(t_n-t_j)\Ph(f(U_1^j)-f(u_1(t_j)))}^2]\right)^{1/2}\,\dd s\\
&\leq C\sum_{j=0}^{n-1}k\left(\E[\norm{U_1^j-u(t_j)}^2]\right)^{1/2}.
\end{align*}
For the second term, using Proposition~\ref{prop:regEx},  we get
\begin{align*}
&\left(\E[\norm{I_{[2,1]}}^2]\right)^{1/2}\\
&\leq\sum_{j=0}^{n-1}\int_{t_j}^{t_{j+1}}\left(\E[\norm{\Lambda_h^{-1/2}S_h(t_n-t_j)\Ph(f(u_1(t_j))-f(u_1(s)))}^2]\right)^{1/2}\,\dd s\\
&\leq C\sum_{j=0}^{n-1}\int_{t_j}^{t_{j+1}}\left(\E[\norm{u_1(t_j)-u_1(s)}^2]\right)^{1/2}\,\dd s\\
&\leq C\sum_{j=0}^{n-1}\int_{t_j}^{t_{j+1}}|t_j-s|^{\min(\beta,1)}\,\dd s\Bigl(\E[\norm{u_0}_\beta^2+\norm{v_0}_{\beta-1}^2]+\sup_{t\in[0,T]}\E[1+\norm{u_1(t)}_{\beta}^2]\Bigr)^{1/2}\\
&\leq Ck^{\min(\beta,1)},
\end{align*}
for $\beta\in[0,3]$. Thus
\begin{align*}
\E[\norm{I_{[2,1]}}^2]&\leq Ck^{2\min(\beta,1)}.
\end{align*}
The second component $I_{[2,2]}$ has the same expression as $I_{[1,2]}$ except that $\Lambda_h^{-1/2}S_h(t_n-t_j)$ is replaced by $C_h(t_n-t_j)$. The same estimate holds since the cosine operator is bounded.

The third term reads, using $\mathcal{K}_h(t)$ in \eqref{initialerror} and $\beta\in[1,3]$,
\begin{align*}
\left(\E[\norm{I_{[3,1]}}^2]\right)^{1/2}&\leq\sum_{j=0}^{n-1}\int_{t_j}^{t_{j+1}}\left(\E[\norm{(\Lambda_h^{-1/2}S_h(t_n-t_j)\Ph-\Lambda^{-1/2}S(t_n-t_j))f(u_1(s))}^2]\right)^{1/2}\,\dd s\\
&=\sum_{j=0}^{n-1}\int_{t_j}^{t_{j+1}}\left(\E[\norm{\mathcal{K}_h(t_n-t_j)f(u_1(s))}^2]\right)^{1/2}\,\dd s\\
&\leq Ch^{\frac{2}{3}\beta}\sum_{j=0}^{n-1}\int_{t_j}^{t_{j+1}}\left(\E[\norm{\Lambda^{(\beta-1)/2}f(u_1(s))}^2\right)^{1/2}\,\dd s\\
&\leq Ch^{\frac{2}{3}\beta}\left(\sup_{t\in[0,T]}\E[1+\norm{u_1(t)}_{\beta-1}^2]\right)^{1/2}\\
&\leq Ch^{\frac{2}{3}\beta}.
\end{align*}
For $\beta\in[0,1]$ we simply note that
$$\E[\norm{\Lambda^{(\beta-1)/2}f(u_1(s))}^2]\leq C\E[\norm{f(u_1(s))}^2]\leq C.$$
The estimate for the second component is done in a similar way using
now $\dot{\mathcal{K}}_h(t)$ in \eqref{initialerror} with $\beta\in[1,4]$,
$$\E[\norm{I_{[3,2]}}^2]\leq Ch^{\frac{4}{3}(\beta-1)},$$

For the fourth term with $\beta\in[0,3]$, using \eqref{41cls} and the
assumption on the function $f$ in \eqref{assFG}, we get  
\begin{align*}
\left(\E[\norm{I_{[4,1]}}^2]\right)^{1/2}&\leq\sum_{j=0}^{n-1}\int_{t_j}^{t_{j+1}}\left(\E[\norm{(S(t_n-t_j)-S(t_n-s))\Lambda^{-1/2}f(u_1(s))}^2]\right)^{1/2}\,\dd s\\
&\leq C\sum_{j=0}^{n-1}\int_{t_j}^{t_{j+1}}\left(\norm{(S(t_n-t_j)-S(t_n-s))\Lambda^{-1/2}}^2_{\mathcal{L}(U)}\E[\norm{f(u_1(s))}^2]\right)^{1/2}\,\dd s\\
&\leq C\sum_{j=0}^{n-1}\int_{t_j}^{t_{j+1}}\left(|s-t_j|^2\E[1+\norm{u_1(s)}^2]\right)^{1/2}\,\dd s\\
&\leq Ck.
\end{align*}
Thus we obtain 
\begin{align*}
\E[\norm{I_{[4,1]}}^2]&\leq Ck^{2}.
\end{align*}
For the second component we get
\begin{align*}
\left(\E[\norm{I_{[4,2]}}^2]\right)^{1/2}&\leq\sum_{j=0}^{n-1}\int_{t_j}^{t_{j+1}}\left(\E[\norm{(C(t_n-t_j)-C(t_n-s))f(u_1(s))}^2]\right)^{1/2}\,\dd s\\
&\leq C\sum_{j=0}^{n-1}\int_{t_j}^{t_{j+1}}\left(\norm{(C(t_n-t_j)-C(t_n-s))\Lambda^{-(\beta-1)/2}}^2_{\mathcal{L}(U)}\right.\\
&\qquad\times\left.\E[\norm{\Lambda^{(\beta-1)/2}f(u_1(s))}^2]\right)^{1/2}\,\dd s\\
&\leq C\sum_{j=0}^{n-1}\int_{t_j}^{t_{j+1}}\left(|s-t_j|^{2(\beta-1)}\E[1+\norm{u_1(s)}_{\beta-1}^2]\right)^{1/2}\,\dd s\\
&\leq Ck^{\min(\beta-1,1)},
\end{align*}
for $\beta\geq 1.$

Altogether we thus obtain
\begin{align*}
\E[\norm{\text{Err}_{\text d,1}}^2]&\leq 
C\cdot\bigl(h^{\frac{4\beta}3}+k^{2\min(\beta,1)}
+k\sum_{j=0}^{n-1}\E[\norm{U_1^j-u_1(t_j)}^2]\bigr)\qquad\text{for}\: \ \beta\in[0,3],\\
\E[\norm{\text{Err}_{\text d,2}}^2]&\leq 
C\cdot\bigl(h^{\frac{4(\beta-1)}3}+k^{2\min(\beta-1,1)}
+k\sum_{j=0}^{n-1}\E[\norm{U_1^j-u_1(t_j)}^2]\bigr)\qquad\text{for}\: \ \beta\in[1,4].
\end{align*}

{\bf Estimate for the stochastic part, $\text{Err}_{\text s}$}. We rewrite the stochastic part
as we did for the deterministic part of the error:
\begin{align*}
\text{Err}_{\text s}&=\sum_{j=0}^{n-1}\int_{t_j}^{t_{j+1}}
\Bigl(E_h(t_n-t_j)\Ph G(U^j)-E(t_n-s)G(X(s))\Bigr)\,\dd W(s)\\
&=\sum_{j=0}^{n-1}\int_{t_j}^{t_{j+1}}E_h(t_n-t_j)\Ph(G(U^j)-G(X(t_j)))\,\dd W(s)\\
&\quad+\sum_{j=0}^{n-1}\int_{t_j}^{t_{j+1}}E_h(t_n-t_j)\Ph(G(X(t_j))-G(X(s)))\,\dd W(s)\\
&\quad+\sum_{j=0}^{n-1}\int_{t_j}^{t_{j+1}}
\bigl(E_h(t_n-t_j)\Ph-E(t_n-t_j)\bigr)G(X(s))\,\dd W(s)\\
&\quad+\sum_{j=0}^{n-1}\int_{t_j}^{t_{j+1}}\bigl(E(t_n-t_j)-E(t_n-s)\bigr)G(X(s))\,\dd W(s)\\
&=:J_1+J_2+J_3+J_4.
\end{align*} 
The estimate for the first term follows by using the Ito isometry, 
the boundedness of $\Ph$, $S_h$ and $\Lambda_h^{-1/2}$, and  
the Lipschitz condition on the function $g$ in \eqref{assFG} 
\begin{align*}
\E[\norm{J_{[1,1]}}^2]&=\sum_{j=0}^{n-1}\int_{t_j}^{t_{j+1}}
\E[\norm{\Lambda_h^{-1/2}S_h(t_n-t_j)\Ph(g(U_1^j)-g(u_1(t_j)))}_{\mathcal{L}_2^0}^2]\,\dd s\\
&\leq Ck\sum_{j=0}^{n-1}\E[\norm{U_1^j-u_1(t_j)}^2]
\end{align*}
for $\beta\in[0,3]$. The same estimate holds for the second component $J_{[1,2]}$ with $\beta\in[1,4]$. 
For the first component of the second term, using Proposition~\ref{prop:regEx}, we obtain
\begin{align*}
\E[\norm{J_{[2,1]}}^2]&=\sum_{j=0}^{n-1}\int_{t_j}^{t_{j+1}}\E[\norm{\Lambda_h^{-1/2}S_h(t_n-t_j)\Ph(g(u_1(t_j))-g(u_1(s))}_{\mathcal{L}_2^0}^2]\,\dd s\\
&\leq C\sum_{j=0}^{n-1}\int_{t_j}^{t_{j+1}}\E[\norm{u_1(t_j)-u_1(s)}^2]\,\dd s\\
&\leq C\sum_{j=0}^{n-1}\int_{t_j}^{t_{j+1}}\abs{t_j-s}^{2\min(\beta,1)}\,\dd s
\Bigl(\E[\norm{u_0}_\beta^2+\norm{v_0}_{\beta-1}^2]\\
&\qquad+\sup_{t\in[0,T]}\E[1+\norm{u_1(t)}_{\beta}^2]\Bigr)\\
&\leq Ck^{2\min(\beta,1)},
\end{align*}
for $\beta\in[0,3]$. Similarly, the estimate for the second component of $J_2$ reads 
\begin{align*}
\E[\norm{J_{[2,2]}}^2]&=\sum_{j=0}^{n-1}\int_{t_j}^{t_{j+1}}
\E[\norm{C_h(t_n-t_j)\Ph(g(u_1(t_j))-g(u_1(s))}_{\mathcal{L}_2^0}^2]\,\dd s\\
&\leq Ck^{2}\Bigl(\E[\norm{u_0}_\beta^2+\norm{v_0}_{\beta-1}^2]+\sup_{t\in[0,T]}\E[1+\norm{u_1(t)}_{\beta}^2]\Bigr).
\end{align*}
For the second component we have $\beta\in[1,4]$, so that $\min(\beta,1)=1$.
For the first component of the third term we use \eqref{initialerror} with $\beta\in(1,3]$ to get
\begin{align*}
\E[\norm{J_{[3,1]}}^2]&=\sum_{j=0}^{n-1}\int_{t_j}^{t_{j+1}}\E[\norm{(\Lambda_h^{-1/2}S_h(t_n-t_j)\Ph-\Lambda^{-1/2}S(t_n-t_j))g(u_1(s))}_{\mathcal{L}_2^0}^2]\,\dd s\\
&\leq Ch^{\frac{4}{3}\beta}\sum_{j=0}^{n-1}\int_{t_j}^{t_{j+1}}\E[\norm{\Lambda^{(\beta-1)/2}g(u_1(s))}^2]\,\dd s\\
&\leq Ch^{\frac{4\beta}{3}}\sup_{t\in[0,T]}\E[1+\norm{u_1(t)}_{\beta}^2]
\leq Ch^{\frac{4\beta}{3}}
\end{align*}
by Proposition~\ref{prop:regEx}. 
The estimate for $\beta\in[0,1]$ is obtained in the same way. For the second component, we also obtain
\begin{align*}
\E[\norm{J_{[3,2]}}^2]&\leq Ch^{\frac{4(\beta-1)}3}\sup_{t\in[0,T]}\E[1+\norm{u_1(t)}_{\beta}^2]
\leq Ch^{\frac{4(\beta-1)}3}
\end{align*}
for $\beta\in[1,4]$. Finally, for the first component of the fourth term, we get
\begin{align*}
\E[\norm{J_{[4,1]}}^2]&=\sum_{j=0}^{n-1}\int_{t_j}^{t_{j+1}}\E[\norm{(S(t_n-t_j)-S(t_n-s))\Lambda^{-\beta/2}\Lambda^{(\beta-1)/2}g(u_1(s))}_{\mathcal{L}_2^0}^2]\,\dd s\\
&\leq C\sum_{j=0}^{n-1}\int_{t_j}^{t_{j+1}}|s-t_j|^{2\beta}\,\dd s\sup_{t\in[0,T]}\E[1+\norm{u_1(t)}^2]\\
&\leq Ck^{2\beta}\sup_{t\in[0,T]}\E[1+\norm{u_1(t)}^2],
\end{align*}
for $\beta\in[0,1]$. For $\beta>1$, we note that $\Lambda^{-\beta/2}=\Lambda^{-1/2}\Lambda^{-(\beta-1)/2}$ 
and that $\Lambda^{-(\beta-1)/2}$ is bounded so that we get
\begin{align*}
\E[\norm{J_{[4,1]}}^2]&\leq Ck^{2\min(\beta,1)}.
\end{align*}
Similarly for the second component, using the regularity of the cosine operator, we obtain
\begin{align*}
\E[\norm{J_{[4,2]}}^2]&\leq Ck^{2\min(\beta-1,1)},
\end{align*}
for $\beta\geq 1$.
Altogether the estimate for the stochastic error reads
\begin{align*}
\E[\norm{\text{Err}_{\text s,1}}^2]&\leq 
C\cdot\Big(h^{\frac{4\beta}3}+k^{2\min(\beta,1)}
+k\sum_{j=0}^{n-1}\E[\norm{U_1^j-u_1(t_j)}^2]\Big)\qquad\text{for}\: \ \beta\in[0,3],\\
\E[\norm{\text{Err}_{\text s,2}}^2]&\leq 
C\cdot\Big(h^{\frac{4(\beta-1)}3}+k^{2\min(\beta-1,1)}
+k\sum_{j=0}^{n-1}\E[\norm{U_1^j-u_1(t_j)}^2]\Big)\qquad\text{for}\: \ \beta\in[1,4].
\end{align*}

Collecting the estimates of the three parts of the error, we thus obtain the 
following estimate for the error in the 
position and velocity of the stochastic wave equation
\begin{align*}
\E[\norm{U_1^n-u_1(t_n)}^2]&\leq C\cdot\Big(h^{\frac{4\beta}3}+k^{2\min(\beta,1)}
+k\sum_{j=0}^{n-1}\E[\norm{U_1^j-u_1(t_j)}^2]\Big),\quad\beta\in[0,3],\\
\E[\norm{U_2^n-u_2(t_n)}^2]&\leq C\cdot\Big(h^{\frac{4}{3}(\beta-1)}+k^{2\min(\beta-1,1)}
+k\sum_{j=0}^{n-1}\E[\norm{U_1^j-u_1(t_j)}^2]\Big),\quad\beta\in[1,4].
\end{align*}
Using the above error bounds and an application of the discrete Gronwall 
lemma proves the following result for the mean-square errors 
of the full discretisation of the semi-linear stochastic 
wave equation with a multiplicative noise. We assume that $\gamma$ is
large enough so that the stochastic error dominates over the initial error.
\begin{theorem}\label{th:ms}
Consider the numerical discretisation of the semi-linear stochastic wave 
equation with a multiplicative noise \eqref{swe} on a compact time interval $[0,T],$ $T>0,$ 
by a linear finite element method in space and 
the stochastic trigonometric method \eqref{stm} in time. 
Assume that $u_0\in L_2(\Omega,\dot{H}^\gamma)$, $v_0\in
L_2(\Omega,\dot{H}^{\gamma-1})$ with $\gamma\ge1+2\beta/3$ and that the functions $f$ and $g$ satisfy 
\eqref{assFG} for some $\beta\geq0$ for the error in the position (and for some $\beta\geq1$ for the error 
in the velocity). Then, for $t_n\in[0,T]$, the mean-square errors read 
\begin{align*}
\norm{U_1^n-u_{h,1}(t_n)}_{L_2(\Omega,\dot{H}^0)}&\leq C\cdot k^{\min(\beta,1)}
\quad\:\text{for}\:\: \ \beta\in[0,2],\\
\norm{U_2^n-u_{h,2}(t_n)}_{L_2(\Omega,\dot{H}^0)}&\leq C\cdot k^{\min(\beta-1,1)}
\quad\:\text{for}\:\: \ \beta\in[1,2],\\
\norm{U_1^n-u_1(t_n)}_{L_2(\Omega,\dot{H}^0)}&\leq C\cdot\bigl(h^{\frac{2\beta}3}+k^{\min(\beta,1)}\bigr)
\quad\:\text{for}\:\: \ \beta\in[0,3],\\
\norm{U_2^n-u_2(t_n)}_{L_2(\Omega,\dot{H}^0)}&\leq C\cdot\bigl(h^{\frac{2(\beta-1)}3}+k^{\min(\beta-1,1)}\bigr)
\quad\:\text{for}\:\: \ \beta\in[1,4].
\end{align*}
\end{theorem} 

Observe that the error estimates between the finite element solutions and the solutions given by the 
stochastic trigonometric method are proven in a similar way as above, using in addition 
Proposition~\ref{prop:regFEM}. 

\section{A trace formula}\label{sect:trace}

In this section, we will only consider the problem \eqref{swe} 
with additive noise ($g\equiv1$ in \eqref{swe}) and 
the nonlinearity $f(u)=-V'(u)$ for a smooth potential $V$. 
We will further consider a trace-class $Q$-Wiener process $W$, 
i.\,e., $\mathrm{Tr}(Q)=\norm{Q^{1/2}}_{\mathrm{HS}}^2<\infty$. 
In this situation, the exact solution of our nonlinear stochastic wave equation 
satisfies a trace formula (see for example \cite{bc01,cls13} for linear 
stochastic wave equations), where, in analogy to deterministic problems, 
the ``Hamiltonian'' function is defined on $H^1=\dot{H}^1\times \dot{H}^0$ as
\begin{align*}
H(X)&=\frac12\int_\mathcal{D}(|u_2|^2+|\nabla u_1|^2)\,\dd x+\int_\mathcal{D} V(u_1)\,\dd x  
\\ &
= \frac12 \|u_{2}\|^2 + \frac12 \| \Lambda^{1/2} u_{1} \|^2 + \int_\mathcal{D}V(u_{1}) \,\dd x.     
\end{align*}
In this section, we restrict our attention to additive noises, since, 
in this case, we obtain an elegant and tractable expression 
for the drift term in the trace formulas (see below). This is not the case 
for the case of multiplicative noise as explained in 
a remark after the proof of the next proposition.

The trace formula for the exact solution to our stochastic wave equation is given in the following 
proposition.
\begin{proposition}\label{TrExact}
Consider the nonlinear stochastic wave equation \eqref{swe} with additive noise, that is with $g\equiv 1$. 
Further, let $f(u)=-V'(u)$ for a smooth potential $V$, let $\{W(t)\}_{t\geq 0}$ be 
a trace-class $Q$-Wiener process, 
and let the Hamiltonian $H$ be defined as above.
Then the exact solution, $X(t)$ in equation \eqref{exactsol}, of the nonlinear stochastic wave equation 
\eqref{swe}, satisfies the trace formula
\begin{align}\label{traEx}
\E[H(X(t))]=\E[H(X(0))]+t\frac{1}{2}\Tr(Q),\quad t\geq 0.
\end{align}
\end{proposition}
\begin{proof}
Indeed, using Ito's formula (one can apply Theorem~4.17 in \cite{DaPrato1992} 
since $X(t)$ is an Ito process and the potential $V$ is smooth enough) 
for the above Hamiltonian, we obtain
\begin{align*}
H(X(t))&=H(X(0))+\int_0^t( H'(X(s)),G\,\dd W(s))+\int_0^t(H'(X(s)),AX+F(X))\,\dd s\\
&\quad+\frac12\int_0^t\Tr[H''(X(s)) (GQ^{1/2})(GQ^{1/2})^*]\,\dd s
\end{align*}
for all time $t$. Here we have $G=\begin{bmatrix}0\\I\end{bmatrix}$, since we are concerned with additive noise. 
The expected value of the second term in the above formula 
is seen to be zero. Using the definition of $A$ and of the nonlinearity $F$, the integrand 
present in the third term reads
$$
(\Lambda u_1,u_2)+(V'(u_1),u_2)+(u_2,-\Lambda u_1-V'(u_1))=0.
$$
Finally, using the above definition of $G$ and the fact that the operator $Q$ is self-adjoint, 
the last term in the above formula is seen to be equal to 
$$
\frac12\int_0^t\Tr(Q^{1/2}(Q^{1/2})^*)\,\dd s=t\frac12\Tr(Q).
$$
This shows the trace formula \eqref{traEx} for the exact solution of our problem.
\end{proof}

\begin{remark}
Similarly to the above computations, for the case of multiplicative noise, one would obtain 
\begin{align*}
H(X(t)) & = H(X(0)) + \int_0^t( H'(X(s)),G(X(s))\,\dd W(s))+\int_0^t(H'(X(s)),AX+F(X))\,\dd s\\
&\quad + \frac12\int_0^t \Tr[H''(X(s)) (G(X(s))Q^{1/2})(G(X(s))Q^{1/2})^*]\,\dd s
\\ & =
H(X(0)) + \int_0^t( H'(X(s)),G(X(s))\,\dd W(s)) 
+
\frac12 \int_0^t \Tr( g(u_1(s)) Q^{1/2} ( g(u_1(s)) Q^{1/2})^*)\,\dd s.
\end{align*}
Taking expectation thus leads to
\begin{align*}
\E[H(X(t))] = \E[H(X(0))]
+
\frac12 \int_0^t \E [ \Tr( g(u_1(s)) Q ( g(u_1(s)) )^*) ]\,\dd s,
\end{align*}
which, for general multiplicative noise, do not give a tractable expression of the drift in the trace formula.
\end{remark}

\begin{remark}
The trace formula is also related to the energy equation, a tool that can be used to analyse 
the existence, or nonexistence, of solutions to stochastic nonlinear wave equations, 
see \cite{chow02} for further details on this topic.
\end{remark}

We next observe that, for the finite element solution $X_h$, one has 
$$
H(X_h) = \frac12 \|u_{h,2}\|^2 + \frac12 \| \Lambda_h^{1/2} u_{h,1} \|^2 + \int_\mathcal{D} V(u_{h,1}) \,\dd x,   
$$
because $\| \nabla v_h \|=\| \Lambda^{1/2} v_h \|=\| \Lambda_h^{1/2} v_h \|$ for 
finite element functions $v_h$. This results from the definitions of 
$\Lambda^{1/2}$ and $\Lambda_h^{1/2}$, see Section~\ref{sect:not}.
Using similar arguments as in the proof of the above result, one can now
show that the finite element solution $X_h(t)$, 
defined in \eqref{exactsolfem}, also possesses a trace formula.
\begin{proposition}\label{TrFEM}
Let $f$, $g$ and $W$ be as in Proposition~\ref{TrExact}. 
The solution of the finite element approximation of problem \eqref{swe}, 
$X_h(t)$ in equation \eqref{exactsolfem}, satisfies the trace formula 
\begin{align}\label{traFEM}
\E[H(X_h(t))]=\E[H(X_h(0))]+t\frac{1}{2}\Tr(\Ph Q \Ph),\quad t\geq 0.
\end{align}
\end{proposition}

We will now prove that the full discretisation of the stochastic wave equation, 
that is the numerical solution given by \eqref{stm}, satisfies an almost trace formula. 
Indeed, as seen in the theorem below, we get a small defect of 
size $\bigo{k^{\min(2(\beta-1),1)}}$.  However, due to the use of
Gronwall's inequality, the defect term is not uniform in time.  

\begin{theorem}\label{TrSTM}
Let $f$, $g$ and $W$ be as in Proposition~\ref{TrExact}~and~\ref{TrFEM}. 
Let further the assumptions in Theorem~\ref{th:ms} be fulfilled 
with $\beta\in[1,2]$.
Then the stochastic trigonometric method \eqref{stm} satisfies an almost trace formula
\begin{align}\label{traFEM-STM}
\E[H(U^n)]=\E[H(U^0)]+t_n\frac{1}{2}\Tr(\Ph Q \Ph)+\bigo{k^{\min(2(\beta-1),1)}}
\end{align}
for $0\leq t_n\leq T$ and $\beta\in[1,2]$. 
\end{theorem}

\begin{proof}
The proof uses similar techniques as the ones used to prove the mean-square error estimates 
for the numerical solution in Section~\ref{sect:ms}. 

To prove the almost trace formula \eqref{traFEM-STM}, we first add and subtract 
the expectation of the Hamiltonian for the finite element solution $X_h(t)$
\begin{align*}
\E[H(U^n)]&=\E[H(U^n)-H(X_h(t_n))]+\E[H(X_h(t_n))]\\
&=\E[H(U^n)-H(X_h(t_n))]+\E[H(X_h(0))]+t_n\frac{1}{2}\Tr(\Ph Q\Ph)
\end{align*}
using Proposition~\ref{TrFEM}. We will next show that
\begin{align}\label{Trace}
\E[H(U^n)-H(X_h(t_n))]=\bigo{k^{\min(2(\beta-1),1)}}
\end{align}
for $\beta\in[1,2]$.
Indeed, we have that 
\begin{align}
\E[H(U^n)-H(X_h(t_n))]&=\E\left[\frac{1}{2}\int_\mathcal{D}(|U_2^n|^2-|u_{h,2}(t_n)|^2)\,\dd x\right.\nonumber\\
&\quad\left.+\frac{1}{2}\int_\mathcal{D}(|\Lambda_h^{1/2}U_1^n|^2-|\Lambda_h^{1/2}u_{h,1}(t_n)|^2)\,\dd x\right.\nonumber\\
&\quad+\left.\int_\mathcal{D}(V(U_1^n)-V(u_{h,1}(t_n)))\,\dd x\right].\label{H3}
\end{align}
Thus we get three terms to estimate. Using Cauchy-Schwarz inequality, 
the first term in the above equation can be estimated by (neglecting 
the factor $\frac12$ for ease of presentation) 
\begin{align*}
\left|\E[\norm{U_2^n}^2-\norm{u_{h,2}(t_n)}^2]\right|&=
\left|\E[(U_2^n+u_{h,2}(t_n),U_2^n-u_{h,2}(t_n))]\right|\\
&\leq \bigl(\E[\norm{U_2^n+u_{h,2}(t_n)}_{h,\beta-1}^2]\bigr)^{1/2} 
\bigl(\E[\norm{U_2^n-u_{h,2}(t_n)}_{h,1-\beta}^2]\bigr)^{1/2}\\
&\leq C\,\bigl(\E[\norm{\Lambda_h^{(1-\beta)/2}(U_2^n-u_{h,2}(t_n))}^2]\bigr)^{1/2},
\end{align*}
where we have used the discrete norm, the fact that the finite element solution $u_{h,2}(t)$ is bounded 
in the mean-square sense (see Proposition~\ref{prop:regFEM}), and the fact that the numerical solution 
given by the stochastic trigonometric method is also bounded, i.\,e. 
\begin{align*}\label{boundTrigo}
\E[\norm{U_1^n}_{h,\beta}^2+\norm{U_2^n}_{h,\beta-1}^2]\leq C<\infty\quad\text{for}\quad n=0,1,\ldots,N-1.
\end{align*}
The proof of these estimates is similar to the one for the finite element solution given in Proposition~\ref{prop:regFEM} 
except that we now have a sum of integrals of length $k$. This causes no problem since we can simply use the triangle inequality for the deterministic integrals and for the stochastic integrals we use the property that they are independent with expected value $0$.

Using the definition of the time integrator and similar techniques as in the proof of the mean-square convergence, 
one next estimates 
\begin{align*}
\Lambda_h^{(1-\beta)/2}(U_2^n-u_{h,2}(t_n))&=
\sum_{j=0}^{n-1}\int_{t_j}^{t_{j+1}}\Lambda_h^{(1-\beta)/2}(C_h(t_n-t_j)-C_h(t_n-s))\Ph\,\dd W(s)\\
&\quad+\sum_{j=0}^{n-1}\int_{t_j}^{t_{j+1}}\Lambda_h^{(1-\beta)/2}C_h(t_n-t_j)\Ph(f(U_1^j)-f(u_{h,1}(t_j)))\,\dd s\\
&\quad+\sum_{j=0}^{n-1}\int_{t_j}^{t_{j+1}}\Lambda_h^{(1-\beta)/2}C_h(t_n-t_j)\Ph(f(u_{h,1}(t_j))-f(u_{h,1}(s)))\,\dd s\\
&\quad+\sum_{j=0}^{n-1}\int_{t_j}^{t_{j+1}}\Lambda_h^{(1-\beta)/2}(C_h(t_n-t_j)-C_h(t_n-s))\Ph f(u_{h,1}(s))\,\dd s\\
&=:J_1+J_2+J_3+J_4.
\end{align*}
Using the temporal regularity of the cosine operator, see \eqref{o41cls}, 
equation \eqref{lpl}, and the assumptions \eqref{assFG} (recall that $g=1$ here), 
one gets
\begin{align*}
\E[\norm{J_1}^2]&=\sum_{j=0}^{n-1}\int_{t_j}^{t_{j+1}}
\norm{\Lambda_h^{(1-\beta)}(C_h(t_n-t_j)-C_h(t_n-s))\Lambda_h^{(\beta-1)/2}\Ph}^2_{\mathcal{L}_2^0}\,\dd s\\
&\leq C\sum_{j=0}^{n-1}\int_{t_j}^{t_{j+1}}
\norm{\Lambda_h^{(1-\beta)}(C_h(t_n-t_j)-C_h(t_n-s))\Lambda_h^{(\beta-1)/2}
\Ph\Lambda^{-(\beta-1)/2}\\
&\qquad\times\Lambda^{(\beta-1)/2}Q^{1/2}}^2_{\text{HS}}\,\dd s\\
&\leq C k^{2\min{(2(\beta-1),1})}\quad\text{for}\quad\beta\in[1,2]. 
\end{align*}
Next, using the convergence results from Theorem~\ref{th:ms} and the Lipschitz assumption on $f$, we observe that 
\begin{align*}
\bigl(\E[\norm{J_2}^2]\bigr)^{1/2}&\leq\sum_{j=0}^{n-1}\int_{t_j}^{t_{j+1}}\bigl(\E[\norm{\Lambda_h^{(1-\beta)/2}C_h(t_n-t_j)\Ph(f(U_1^j)-f(u_{h,1}(t_j)))}^2]\bigr)^{1/2}\,\dd s\\
&\leq C k\sum_{j=0}^{n-1} \bigl(\E[\norm{U_1^j-u_{h,1}(t_j)}^2]\bigr)^{1/2}\leq C k^{\min(\beta,1)}\quad\text{for}\quad\beta\in[1,2]. 
\end{align*}
Similarly, using the assumptions on $f$ given in \eqref{assFG}, 
and the regularity property of the finite element solution stated in Proposition~\ref{prop:regFEM}, one gets
\begin{align*}
\bigl(\E[\norm{J_3}^2]\bigr)^{1/2}&\leq\sum_{j=0}^{n-1}\int_{t_j}^{t_{j+1}}
\bigl(\E[\norm{\Lambda_h^{(1-\beta)/2}C_h(t_n-t_j)\Ph(f(u_{h,1}(t_j))-f(u_{h,1}(s)))}^2]\bigr)^{1/2}\,\dd s\\
&\leq C k^{\min(\beta,1)}\quad\text{for}\quad\beta\in[1,2]. 
\end{align*}
For the last term, $J_4$, we obtain the estimate for $\beta\in[1,2]$ as follows
\begin{align*}
\bigl(\E[\norm{J_4}^2]\bigr)^{1/2}&\leq\sum_{j=0}^{n-1}\int_{t_j}^{t_{j+1}}
\bigl(\E[\norm{\Lambda_h^{(1-\beta)}(C_h(t_n-t_j)-C_h(t_n-s))\\
&\qquad\times\Lambda_h^{(\beta-1)/2}\Ph f(u_{h,1}(s))}^2]\bigr)^{1/2}\,\dd s \\
&\leq C k^{\min(2(\beta-1),1)},
\end{align*}
where we have used equation \eqref{o41cls}, the equivalence between the norms \eqref{equivnorm}, 
the assumptions on the nonlinearity $f$, and the fact that the finite element solution $u_{h,1}$ 
is bounded in the norm $\norm{\cdot}_{h,\beta-1}$. 

Collecting all the above estimates and observing that $2(\beta-1)\leq\beta$ for $\beta\in[1,2]$, we finally get 
\begin{align*}
\left|\E[\norm{U_2^n}^2-\norm{u_{h,2}(t_n)}^2]\right| \leq C k^{\min{(2(\beta-1),1)}}\quad\text{for}\quad\beta\in[1,2].
\end{align*}
The second term in \eqref{H3} can be estimated in a similar way as above. We have
\begin{align*}
\left|\E[\norm{\Lambda_h^{1/2}U_1^n}^2-\norm{\Lambda_h^{1/2}u_{h,1}(t_n)}^2] \right| &\leq(\E[\norm{U_1^n+u_{h,1}(t_n)}_{h,\beta}^2])^{1/2}(\E[\norm{U_1^n-u_{h,1}(t_n)}_{h,2-\beta}^2])^{1/2}\\
&\leq C(\E[\norm{\Lambda_h^{(2-\beta)/2}(U_1^n-u_{1,h}(t_n))}^2])^{1/2}\\
&\leq C k^{\min{(2(\beta-1),1)}}\quad\text{for}\quad\beta\in[1,2].
\end{align*}
When estimating $\E[\norm{\Lambda_h^{(2-\beta)/2}(U_1^n-u_{1,h}(t_n))}^2]$ we get the same terms as $J_1$ through $J_4$ above, 
except that cosine is replaced by sine everywhere. Hence the same estimate holds.
For the third and final term in \eqref{H3}, using the mean value theorem we get first 
\begin{align*}
\E[\norm{V(U_1^n)-V(u_{h,1}(t_n))}_{L_1(\mathcal{D})}]&\leq C\E[\norm{V(U_1^n)-V(u_{h,1}(t_n))}_{L_2(\mathcal{D})}]\\
&\leq C\norm{V'(\xi)(U_1^n-u_{h,1}(t_n))}_{L_2(\Omega,\dot H^0)}.
\end{align*}
Recalling that $f(u)=-V'(u)$, using H\"older's inequality, 
using the fact the numerical solutions are bounded in the mean-square sense, 
and the error bounds stated in Theorem~\ref{th:ms}, we next estimate the following expression
\begin{align*}
\E[\norm{V(U_1^n)-V(u_{h,1}(t_n))}_{L_1(\mathcal{D})}]&\leq 
C\norm{V'(\xi)(U_1^n-u_{h,1}(t_n))}_{L_2(\Omega,\dot H^0)}\\
&\leq C \left(\E[\norm{U_1^n-u_{h,1}(t_n)}_{L_2(\mathcal{D})}^2]\right)^{1/2}\leq C k^{\min(\beta,1)}.
\end{align*}
Putting all these estimates together 
we obtain equation \eqref{Trace} and the theorem is proven.
\end{proof}


\section{Numerical experiments}\label{sect:num}
This section illustrates numerically the main results of the paper. 
We first present the time integrators we will consider, 
then test their mean-square orders of convergence on various problems 
and finally illustrate their behaviours with respect to the trace formula from the previous section.
\subsection{Setting}
The solution of our stochastic wave equation \eqref{swe} will now be 
numerically approximated using the method of lines, i.\,e., 
with a linear finite element method in space and then with 
various time integrators (see below). Further, we will consider two kinds of noise: 
a space-time white noise with covariance operator $Q=I$ and a correlated one 
with $Q=\Lambda^{-s}$ for some $s>0$. We refer for example to \cite{cls13} for 
a discussion on the approximation of the noise. 
 
We shall compare the stochastic trigonometric method \eqref{stm} with the following 
classical numerical schemes for stochastic differential equations. 
When applied to the wave equation in the form \eqref{swe2}, these numerical integrators are: 
\begin{enumerate}
\item The forward Euler-Maruyama scheme, see for example \cite{kp92} or \cite{mt04},  
$$
X^{n+1}=X^n+kAX^{n}+kF(X^{n})+G(X^n)\Delta W^n.
$$ 
\item The semi-implicit Euler-Maruyama scheme, see for example \cite{h03} or \cite{wal05},  
$$
X^{n+1}=X^n+kAX^{n+1}+kF(X^{n})+G(X^n)\Delta W^n.
$$
\item The backward Euler-Maruyama scheme, see for example \cite{kp92} or \cite{mt04},  
$$
X^{n+1}=X^n+kAX^{n+1}+kF(X^{n+1})+G(X^n)\Delta W^n.
$$ 
\item The semi-implicit Crank-Nicolson-Maruyama scheme, \cite{h03} or \cite{wal05}, 
$$
X^{n+1}=X^{n}+\frac k2 A(X^{n+1}+X^{n})+kF(X^n)+G(X^n)\Delta W^n.
$$
\end{enumerate}
Note that the backward Euler-Maruyama scheme, the semi-implicit Euler-Maruyama scheme, and 
the semi-implicit Crank-Nicolson-Maruyama scheme are implicit numerical integrators. 

All the numerical experiments were performed in Matlab using specially designed software 
and the random numbers were generated with the command \verb+randn+. 

\subsection{Multiplicative noise}
Let us first consider the one-dimensional 
hyperbolic Anderson model \cite{cd08,dm09}:
\begin{equation*}
\begin{aligned}
&\mathrm{d}\dot{u}(x,t)-u_{xx}(x,t)\, 
\mathrm{d}t = u(x,t)\,\mathrm{d}W(x,t) &&\quad \mathrm{for}\quad (x,t) \in \ (0,1)\times(0,1), \\
&u(0,t)=u(1,t) = 0, &&\quad t\in(0,1), \\
&u(x,0) = \sin(2\pi x), \ \dot u(x,0)=\sin(3\pi x), &&\quad x \in(0,1).
\end{aligned}
\end{equation*}
This stochastic partial differential equation with multiplicative noise 
is now discretised in space by a linear finite element method with mesh size $h$. 
This leads to a system of stiff stochastic differential equations. 
The latter problem is then discretised in time by various integrators with time step $k$. 

Figure~\ref{fig:errorAndersonh} illustrates the results on the spatial discretisation 
of the finite element method as stated in Theorem~\ref{th:ms}.   
The spatial mean-square errors at time $T_{\mathrm{end}}=1$,
$$
\sqrt{\E\big[\norm{u_h(x,T_{\mathrm{end}})-u(x,T_{\mathrm{end}})}^2\big]},
$$ 
are displayed for various values of the parameter $h=2^{-\ell}$, $\ell=2,\ldots,9$. 
The covariance operator is chosen as $Q=\Lambda^{-s}$ for $s=0,1/2,1/3,1/4$. 
In the present situation, $f(u)=0$ and $g(u)=u$ satisfy the assumptions 
\eqref{assFG} with $\beta<s+\frac12$. This can be seen using the computations done in 
Subsection~4.1 from \cite{jr12} (with $\rho=2s$ and $\alpha=\frac{\beta-1}2)$. 
A clear dependence of the spatial convergence rates with respect 
to the covariance operator can be observed in this figure, in
agreement with Theorem~\ref{th:ms}.  
Here, we simulate the exact solution $u(x,t)$ with the numerical one  
using the stochastic trigonometric method (STM) \eqref{stm} with a small time step 
$k_{\mathrm{exact}}=2^{-9}$ (in order to neglect the error from the discretisation in time) 
and $h_{\mathrm{exact}}=2^{-9}$ for the mesh of the FEM. 
The expected values are approximated by computing averages 
over $M_{\mathrm{s}}=2500$ samples. We computed the estimate 
for the largest standard errors, of all schemes, to be $0.0026$. 
This shows that the error due to a Monte-Carlo approximation is negligible.

\begin{figure}
\begin{center}
\includegraphics*[height=6cm,keepaspectratio]{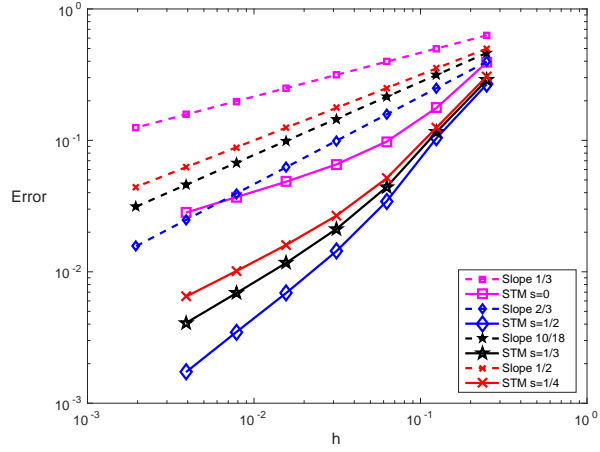}
\caption{The Anderson model: Spatial rates of convergence for the covariance operators 
$Q=\Lambda^{-s}$ with $s=0,1/2,1/3,1/4$. The dotted lines are reference 
lines of slopes $1/3,2/3,10/18,1/2$. $M_{\mathrm{s}}=2500$ samples.} 
\label{fig:errorAndersonh}
\end{center}
\end{figure}

We are now interested in the time discretisation of the above 
stochastic partial differential equation with space-time white noise 
($Q=I$ and thus $\beta<1/2$). We compute the temporal errors at time 
$T_{\mathrm{end}}=0.5$. In Figure~\ref{fig:errorAndersonkall}, 
one can observe the rates  
of mean-square convergence of various time integrators. 
The expected rate of convergence $\bigo{k^{1/2}}$ of the stochastic trigonometric 
method as stated in Theorem~\ref{th:ms} can be confirmed. Again, the exact solution 
is approximated by the stochastic trigonometric method 
with a very small time step $k_{\mathrm{exact}}=2^{-11}$ and uses  
$h_{\mathrm{exact}}=2^{-9}$ for the spatial discretisation. 
$M_{\mathrm{s}}=2500$ samples are used for the approximation of the expected values. 
We computed the estimate for the largest standard errors, of all schemes, to be $0.01$. 
The numerical results for the forward and backward Euler-Maruyama schemes are not 
displayed since these numerical schemes would have to use very small time steps for 
such an $h_{\mathrm{exact}}$ (see also Subsection~\ref{numexpTrace} below).

\begin{figure}
\begin{center}
\includegraphics*[height=6cm,keepaspectratio]{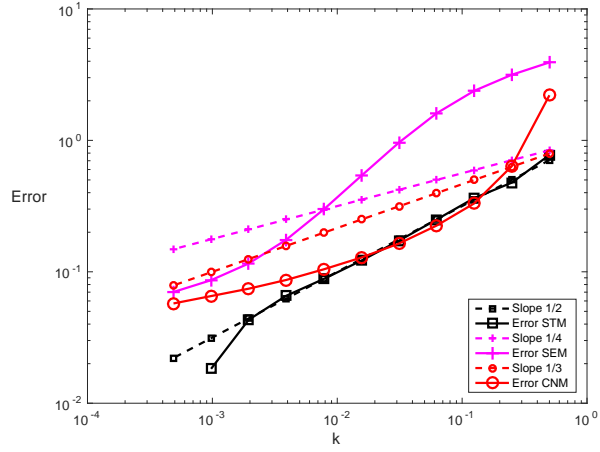}
\caption{The Anderson model (space-time white noise): Temporal rates of convergence of the stochastic trigonometric method (STM), 
the semi-implicit Euler-Maruyama scheme (SEM) and the Crank-Nicolson-Maruyama scheme (CNM). 
The reference lines have slopes $1/4,1/3$ and $1/2$. $M_{\mathrm{s}}=2500$ samples.}
\label{fig:errorAndersonkall}
\end{center}
\end{figure}

\subsection{Semi-linear problem with additive space-time white noise}
We next consider the sine-Gordon equation driven by additive space-time white noise 
($Q=I$ and thus $\beta<1/2$)
\begin{equation*}
\begin{aligned}
&\mathrm{d}\dot{u}(x,t) - u_{xx}(x,t)\, \mathrm{d}t = -\sin(u(x,t))\, 
\mathrm{d}t+\mathrm{d}W(x,t), &&\quad (x,t) \in \ (0,1)\times(0,0.5), \\
&u(0,t)=u(1,t) = 0, &&\quad t\in(0,0.5), \\
&u(x,0) = 0, \ \dot u(x,0)=1_{[\frac{1}{4},\frac{3}{4}]}(x), &&\quad x \in(0,1),
\end{aligned}
\end{equation*}
where $1_I (x)$ denotes the indicator function for the interval $I$. 

Figure~\ref{fig:errorSGall} displays the rates  
of mean-square convergence at $T_{\mathrm{end}}=0.5$ of various time integrators. 
The expected temporal rate of convergence $\bigo{k^{1/2}}$ of the stochastic trigonometric 
method as stated in Theorem~\ref{th:ms} can be confirmed. Again, the exact solution 
is approximated by the stochastic trigonometric method 
with a very small step size $k_{\mathrm{exact}}=2^{-11}$ and uses  
$h_{\mathrm{exact}}=2^{-9}$ for the spatial discretisation. 
$M_{\mathrm{s}}=2500$ samples are used for the approximation 
of the expected values. We computed the estimate 
for the largest standard errors for all schemes to be $0.0027$, showing 
that the error due to a Monte-Carlo approximation 
is negligible.

\begin{figure}
\begin{center}
\includegraphics*[height=6cm,keepaspectratio]{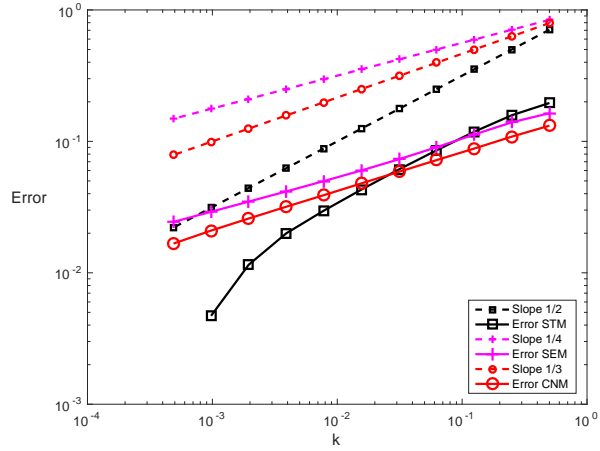}
\caption{The sine-Gordon equation (space-time white noise): Temporal rates of convergence 
of the stochastic trigonometric method (STM), 
the semi-implicit Euler-Maruyama scheme (SEM) and the Crank-Nicolson-Maruyama scheme (CNM). 
The dotted lines have slopes $1/4,1/3$ and $1/2$. 
$M_{\mathrm{s}}=2500$ samples.}
\label{fig:errorSGall}
\end{center}
\end{figure}

\subsection{Semi-linear equation with multiplicative noise}
In this subsection, we consider the sine-Gordon equation 
driven by a multiplicative space-time white noise ($Q=I$ and thus $\beta<1/2$)
\begin{equation*}
\begin{aligned}
&\mathrm{d}\dot{u}(x,t) - u_{xx}(x,t)\, \mathrm{d}t = 
-\sin(u(x,t))\, \mathrm{d}t+u(x,t)\,\mathrm{d}W(x,t), &&\quad (x,t) 
\in \ (0,1)\times(0,0.5), \\
&u(0,t)=u(1,t) = 0, &&\quad t\in(0,0.5), \\
&u(x,0) = \sin(2\pi x), \ \dot u(x,0)=\sin(3\pi x), &&\quad x \in(0,1).
\end{aligned}
\end{equation*}

Figure~\ref{fig:errorMSG} displays the rates of mean-square convergence 
of various time integrators when applied to this semi-linear problem with multiplicative noise. 
The expected temporal rate of convergence $\bigo{k^{1/2}}$ of the stochastic trigonometric 
method as stated in Theorem~\ref{th:ms} can be confirmed. 
One also observes a slower convergence 
rate for the other integrators. As before, a reference solution 
is computed by the stochastic trigonometric method 
with a very small step size $k_{\mathrm{exact}}=2^{-11}$ and uses  
$h_{\mathrm{exact}}=2^{-9}$ for the spatial discretisation. 
$M_{\mathrm{s}}=2500$ samples are used 
for the approximation of the expected values. We computed the estimate 
for the largest standard errors, of all schemes, to be $0.006$. 
This shows that the error due to a Monte-Carlo approximation is negligible. 

\begin{figure}
\begin{center}
\includegraphics*[height=6cm,keepaspectratio]{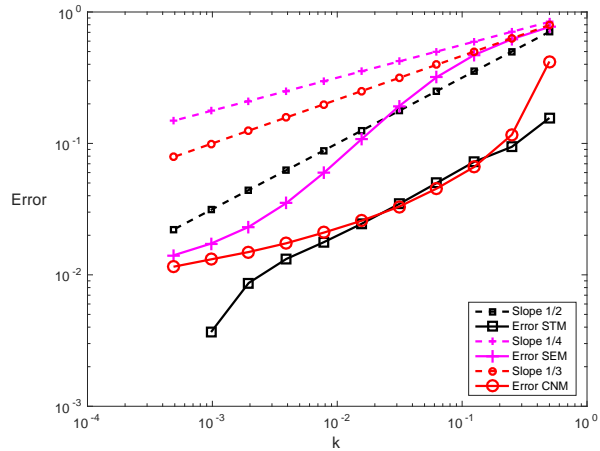}
\caption{The sine-Gordon equation with multiplicative space-time white noise: 
Temporal rates of convergence of the stochastic trigonometric method (STM), 
the semi-implicit Euler-Maruyama scheme (SEM) and the Crank-Nicolson-Maruyama scheme (CNM). 
The dotted lines have slopes $1/4,1/3$ and $1/2$. $M_{\mathrm{s}}=2500$ samples.}
\label{fig:errorMSG}
\end{center}
\end{figure}

\subsection{Trace formula}\label{numexpTrace}
We will now illustrate the trace formula from Section~\ref{sect:trace}. 
To do this, we again consider the above sine-Gordon equation with additive noise 
and solve this problem with a linear finite element method in space and in time 
we use the stochastic trigonometric method \eqref{stm} with $f(u)=-\sin(u)$, $g(u)=1$. 
Figure~\ref{fig:sine_gordon} (top) displays the expected value of 
the Hamiltonian along the numerical solutions of the above 
stochastic sine-Gordon equation where 
the covariance operator is given by $Q=\Lambda^{-2}$. In the present situation, 
the Lipschitz function $f(u)=-\sin(u)$ and the function $g(u)=1$ satisfy 
the assumptions \eqref{assFG} with $\beta=2$. This is seen using the fact that 
the eigenvalues of the Laplace operator with Dirichlet boundary condition satisfy 
$\lambda_j\sim j^2$ and the eigenvectors are given by $\{\sqrt{2}\sin(j\pi x)\}_j$.  
The meshes are $h=0.1$ and $k=0.01$, the time interval is $[0,5]$, 
and $M_{\mathrm{s}}=2500$ samples are used for the 
approximation of the expected values. For this experiment, 
the largest standard errors for all the numerical schemes  
(except for the Euler-Maruyama scheme)  
is of the size of $0.002$ confirming that the Monte-Carlo errors are negligible. 
In this figure, one can observe the unsatisfactory behaviour of classical Euler-Maruyama-type methods. 
This is not a big surprise, since, already for stochastic ordinary differential equations, 
the growth rate of the expected energy along solutions given by these numerical solutions 
is incorrect \cite{hig04,c10}. The Crank-Nicolson-Maruyama scheme however seems to reproduce 
very well the linear drift in the expected value of the Hamiltonian. 
Let us see what happens when one uses bigger time step and longer time interval. 
Figure~\ref{fig:sine_gordon} (bottom) displays the expected energies on the 
longer time interval $[0,250]$ for the Crank-Nicolson-Maruyama and the stochastic trigonometric methods 
with a larger time step $k=0.1$. The other parameters are the same as in the above numerical experiment, in particular, 
the Monte-Carlo error for the stochastic trigonometric method is negligible 
(error of size $0.01$). On this long-time interval, excellent behaviour of 
the stochastic trigonometric method \eqref{stm} is still observed although 
this does not follow from the result presented in Theorem~\ref{TrSTM}. 

\begin{figure}
\begin{center}
\includegraphics*[height=6cm,keepaspectratio]{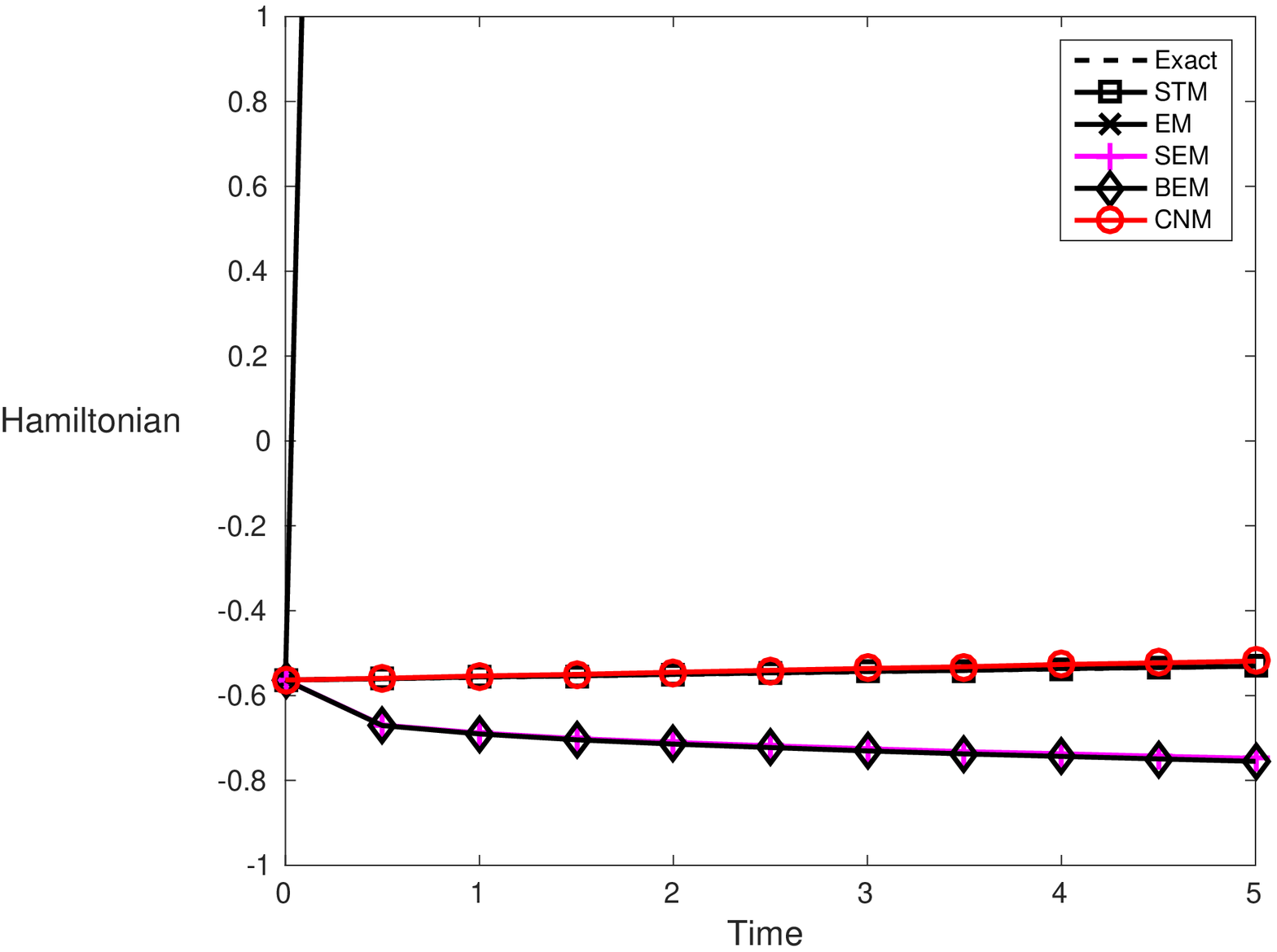}
\includegraphics*[height=6cm,keepaspectratio]{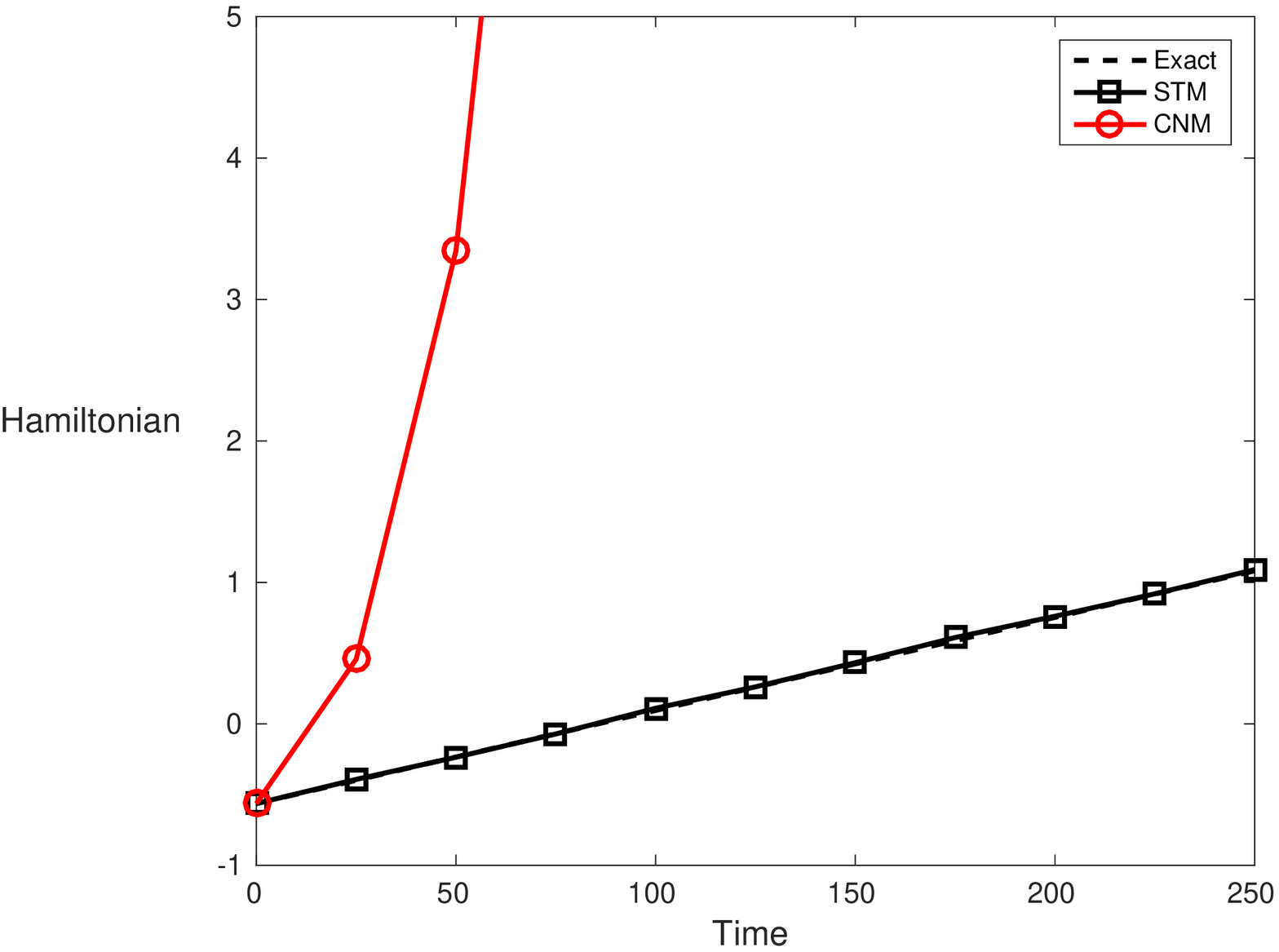}
\caption{Trace formula for the sine-Gordon equation: Expected values of the Hamiltonian along the numerical 
solutions given by the stochastic trigonometric method (STM), the forward Euler-Maruyama scheme (EM), 
the semi-implicit Euler-Maruyama scheme (SEM), the backward Euler-Maruyama scheme (BEM) and the 
Crank-Nicolson-Maruyama scheme (CNM). $M_{\mathrm{s}}=2500$ samples. 
Mesh size: $h=0.1$. Time intervals and time steps: $[0,5]$, $k=0.01$ (top) 
and $[0,250]$, $k=0.1$ (bottom).}
\label{fig:sine_gordon}
\end{center}
\end{figure}


\section{Appendix}
In order to improve the readability of the paper, we give some details for the proofs of 
the results given in Section~\ref{sect:not}. 
The proofs of \eqref{initialerror1} and \eqref{initialerror} can be
found in Corollary~4.2 in \cite{kls10}.  They are obtained by
interpolation between the results for the endpoints of the parameter
values.  These in turn are well-known estimates for the finite element
approximation of the homogeneous wave equation:
$ \ddot{u}_h+\Lambda_h u_h=0, \ t\ge0; \ u_h(0)=\Rh u_0,\
\dot{u}_h(0)=\Ph v_0$.

For example, for $\gamma=1$ we have, by a standard stability estimate,
\begin{align*}
\norm{\mathcal{G}_h(t)X_0}&\leq C\left(\norm{\Rh u_0}+\norm{\Ph v_0}_{-1,h}+\norm{u_0}+\norm{v_0}_{-1}\right)\\
&\leq C\left(\norm{u_0}_1+\norm{v_0}_{-1}\right)\le C h^0\tnorm{X_0}_1,
\end{align*}
since $\Rh$ is not bounded with respect to the $\dot{H}^0$-norm.  
For $\gamma=3$, we have
\begin{align*}
\norm{\mathcal{G}_h(t)X_0}
\leq C(1+t) h^2\left(\norm{u_0}_3+\norm{v_0}_2\right)
&\leq C(1+t)h^2\tnorm{X_0}_3,
\end{align*}
cf.\ the estimation of $F_h$ in the proof of Corollary~4.2 in \cite{kls10}. 
Interpolation between these two cases completes the proof of the first
bound in \eqref{initialerror1}.  Note that the required initial
regularity is one order higher than the order of convergence.  This is
typical of the finite element method for the wave equation.  Another
choice of projector, $u_h(0)=\Ph u_0$, would give a slightly better
result for low initial regularity here, but a worse result for
$\dot{\mathcal{G}}_h$.


\bibliographystyle{siam}
\bibliography{biblio}

\end{document}